\documentclass[11pt,reqno]{amsart}
\usepackage{amscd,amsmath,amstext,amsfonts,amsbsy,amssymb,amsthm,mathrsfs,float}
\usepackage{multicol,graphicx}
\usepackage{array,multirow}
\usepackage{fancyhdr,booktabs}
\usepackage{amsfonts}
\usepackage{amscd,amsmath,amstext,amsfonts,amsbsy,amssymb,amsthm,mathrsfs,float}
\usepackage{multicol,graphicx}
\usepackage{array,multirow}
\usepackage{fancyhdr,booktabs}
\usepackage{hyperref}
\usepackage{enumerate,fullpage}
\usepackage{amssymb,amsmath,graphicx,amsthm}
\usepackage{amsmath,amsfonts,latexsym,amsthm,amssymb}
\usepackage{color}
\usepackage{amssymb}
\usepackage{geometry}
\usepackage{color}
\usepackage{enumerate,fullpage}
\usepackage{amssymb,amsmath,graphicx,amsthm}
\usepackage{amsmath,amsfonts,latexsym,amsthm,amssymb}
\newcommand{\inpro}[2]{\left\langle{#1},{#2}\right\rangle}
\newtheorem{theorem}{{\bf Theorem}}[section]
\theoremstyle{definition} \newtheorem{definition}[theorem]{\bf Definition}

\theoremstyle{plain} \newtheorem{lemma}[theorem]{Lemma}

\newtheorem{remark}{Remark}[section]

\usepackage{amssymb}
\usepackage{bbm}
\usepackage{anysize}
\marginsize{1in}{1in}{1in}{1in}

\newcounter{smalllist}

\newcommand{\lf}{\left}
\newcommand{\rt}{\right}
\newcommand{\be}{\begin{eqnarray*}}
	\newcommand{\en}{\end{eqnarray*}}
\newcommand{\bes}{\begin{eqnarray}}
\newcommand{\ens}{\end{eqnarray}}

\newcommand{\ep}{\epsilon}
\def\nn{\nonumber}
\def\leqslant {\le}
\def\bq{\begin{equation}}
\def\eq{\end{equation}}
\def\bqq{\begin{eqnarray*}}
	\def\eqq{\end{eqnarray*}}
\title[Approximate  solutions for backward problem]{ Approximate  solutions of  inverse problems for nonlinear space fractional diffusion equations with   randomly perturbed data   }

\author[E. Nane]{Erkan Nane}
\address[E. Nane]{Department of Mathematics and Statistics, Auburn University, Auburn, USA}%
\email{ezn0001@auburn.edu}

\author[N.H. Tuan]{Nguyen Huy Tuan}
\address[N.H. Tuan]{ Department of Mathematics and Computer Science, University of  Science, Vietnam National University HCMC,  Ho Chi Minh City, Viet Nam}
\email{thnguyen2683@gmail.com}
\begin{document}
	
	\begin{abstract}
	This paper is concerned with backward problem for nonlinear space fractional diffusion  with additive  noise on the right-hand side and the final value.
	To regularize the instable solution, we  develop some new regularized method for solving the problem. In the case of constant coefficients, we use the truncation methods. In the case of perturbed time dependent coefficients, we apply a new quasi-reversibility method.  We also show the convergence rate  between the regularized solution and the sought solution under some a priori assumption on the sought solution.
		
	\end{abstract}

	\maketitle
	\tableofcontents

\section{Introduction}

 In this paper we focus on the problem  of finding  the initial  function $u(x,0)=u_0(x)$ such that  $u(x,t),\  t>0$ satisfies the following  final value problem for the nonlinear equation with fractional Laplacian
\begin{equation}
\label{problem2}
\left\{\begin{array}{l l l}
u_t + a(t)  (-\Delta )^\beta u & = F(u)+g(x,t), & \qquad (x,t) \in \Omega\times (0,T),\\
u_x(x,t)&=0, & \qquad x \in \partial {\Omega},\\
u(x,T) & = u_T(x), & \qquad x \in {\Omega},
\end{array}
\right.
\end{equation}
where {$\beta >\frac{1}{2}$
 is a given constant,  see equation  \eqref{beta-condition} for this restriction on $\beta$}. The domain $\Omega=(0,\pi)$ is 1-D domain.  The function $F$ and $g$ are called the source functions which will be defined later. $a(t)$ is a given time dependent coefficient.   The function $u_T$ is called the final value data. $(-\Delta )^\beta$ is the fractional Laplacian that will be explained in Section 2. The space and time  fractional diffusion  has been studied recently in \cite{Nane}. In this paper, we only consider the problem  with fractional order of space variable  defined by spectral theory. Our fractional Laplacian in this paper differs from the fractional Laplacian defined in \cite{Nane}.

The problem \eqref{problem2} with $\beta =a(t)= 1$  is  the backward problem for classical parabolic equation.
These problems are applied in fields  such as the heat conduction theory \cite{BBC}, material science \cite{RHN}, hydrology \cite{JB,LEP}, groundwater contamination \cite{SK}, digital remove blurred noiseless image \cite{CSH} and also in many other practical applications of mathematical physics and engineering. It is well--known that  the backward parabolic  problem is severely ill--posed (see \cite{r5}). The  solutions do not always exist, and in the case of existence,
the solutions  do not depend continuously on the given initial  data. In fact, from small noise contaminated physical measurements,
the corresponding solutions will  have large errors. { Therefore, some regularized methods are required to find approximate solution. If $\beta=1$, the deterministic case for Problem \eqref{problem2} has been studied by \cite{PTN,Tuan2,Tuan3}.}

The analysis
of regularization methods for the stable solution of Problem \eqref{problem2} depends on the mathematical
model for the noise term on the source function $g$ and  the final value data $u_T$.  We
suppose that the measurements are described by functions
\begin{equation}\label{observed-data}
g^{\text{obs}}= g+ \text{"noise"}, \quad u_T^{\text{obs}}= u_T+ \text{"noise"}.
\end{equation}

 If the noise is considered as a
deterministic quantity, it is natural to study the worst-case error. In the literature a
number of efficient methods for the solution of \eqref{problem2} have been developed: see, for example, {\cite{Bissantz,Koba},} and the references therein.

 If the errors are generated from uncontrollable sources such  as wind, rain, humidity, etc, then the model is random.
 If the noise is
 modeled as a random quantity, the convergence of estimators  $\widetilde u(x,0)$ of  $u(x,0)$ should be studied
 in statistical terms. More details on ill-posedness of the problem \eqref{problem2} in the case of $F=0, \beta=1$ with random noise can be found in \cite{Minh}.  Methods for the deterministic cases cannot apply directly to this case.   Our main purpose in the random noise case  is  finding suitable  estimators  $\widetilde u(x,0)$ of $ u(x,0)$ and consider the expected square error $\mathbb{E} \|\widetilde u(x,0) - u(x,0)\|$, also  called the mean integrated square error (MISE).

There exist a considerable amount of literature on regularization methods for
linear backward problem  with random noise.
%In our knowledge, we can list here some related papers.
In  Cavalier  \cite{Cavalier}, the author gave some theoretical examples about inverse problems with random noise. Mair  and Ruymgaart  \cite{r7} considered theoretical formulae  for statistical inverse estimation in Hilbert spaces
and applied the method to  some examples. Recently, Hohage et al. \cite{Hohage} applied  spectral cut-off (truncation method) and Tikhonov-type methods for solving linear statistical inverse problems including backward heat equation (See p. 2625, \cite{Hohage}). In the linear inhomogenous case of \eqref{problem2}, i.e, $\beta=1$ and $F=0$, the Problem \eqref{problem2} has been recently  studied in \cite{Minh} in two  space dimensions.

To the best of the  authors' knowledge,  the backward problem for nonlinear  parabolic  equation with random noise was not
investigated in the literature. This is  one of the motivations of our present paper.
Next we discuss the difficulty of investigating  the nonlinear problem. A well--known fact is the following:  if  $F(u)=0$ then the problem \eqref{problem2} and \eqref{observed-data} can be transformed into a linear operator with random case
\begin{equation}
u_T=C u_0 + \text{"noise"}.
\end{equation}
Where  {$C$  is a linear bounded operator with an unbounded inverse}.
There are  many well-known methods developed by  Cavalier \cite{Cavalier}, Hohage et al \cite{Hohage}, Siltanen \cite{Lassas}, Trong et al \cite{Trong}  as above,  for solving the latter linear model. However, when $F$ depends on $u$, we can not transform Problem \eqref{problem2} into a  linear one, this makes the nonlinear problem  more difficult to study.
Therefore, we have to develop some new methods to solve the nonlinear problem.\\

In this paper, using a similar random model  given in \cite{Minh}, we consider the nonlinear problem  as follows
 \begin{equation}
 	\widetilde u_T(x_k)=u_T(x_k)+\sigma_k\epsilon_k, \quad  \widetilde g_k(t)= g(x_k,t)+ \vartheta\xi_{k}(t),\quad \text{for} \quad  k=\overline{1,n},  \label{noise}
 \end{equation}
where $x_k = \pi\dfrac{2k-1}{2n}$
  and  $\epsilon_k$ are unknown independent random errors. Moreover,   $\epsilon_{k}\sim \mathcal{N}(0,1)$,  and $\sigma_{k} $ are unknown positive constants which are bounded by a positive constant $V_{max}$, i.e., $0 \leq \sigma_{k} < V_{\text{max}}$ for all $k=1,\cdots, n$.  $\xi_{k}(t)$'s are Brownian motions. The noises $ \epsilon_{k}, \xi_{k}(t)$ are mutually independent. {A similar model with noise in equation  \eqref{noise} without  the $g$ part has been recently considered by Tuan and Nane \cite{Tuan1}}. \\

 Next we give some details about our methods  for  the following two cases

\noindent {\bf The first case: $a(t)=1$}.
First, we transform the problem \eqref{problem2} into a nonlinear integral equation, then we apply the Fourier truncation method(using the eigenfunctions $\cos (px)$, $p=0,1,2\cdots$ of the Laplacian in the interval $(0,\pi)$ with Neumann boundary conditions) associated with some techniques in nonparametric regression to establish a first regularized solution $\overline U_{M_n,n}(x,t)$ which satisfies  \eqref{resol}. To obtain the estimate between  $\overline U_{M_n,n}$ and $u$, we need some stronger assumptions on $u$, such as \eqref{ass1} and \eqref{ass2}.  {\bf The main result is Theorem \ref{main-theorem-1}.} However, as pointed out in Remark \ref{remark2},  the assumptions \eqref{ass1} and \eqref{ass2} are  difficult to come up in practice. Motivated by this, when $g=0$ we develop a second regularized solution  $\widehat	U_{M_n,n}$ defined by \eqref{resol2} to obtain the estimate for
$u \in C([0,T];H^{\gamma}(\Omega))$ (see  Remark \ref{remark2} for more details). {\bf The main result for the second type of regularization is given in Theorem \ref{main-theorem-2}.} It is important to realize that  the second regularized solution is a modification of the first regularized solution. { Our methods in this paper can be applied to solve many ill-posed problems of nonlinear PDEs such as Cauchy problem for nonlinear elliptic, nonlinear ultraparabolic, nonlinear strongly damped wave equations, and many others}.% Our main results in this case are Theorems  \ref{main-theorem-1} and \ref{main-theorem-2}.

 {
\noindent {\bf The second case: $a(t)$ depend on $t$ and is perturbed}.
Note that if the coefficient $a$ in the main equation of \eqref{problem2}  is not noisy then the Fourier truncation method in \cite{Tuan2} can be applied to Problem \eqref{problem2}.   However, the difficulty  occurs for \eqref{problem444} when the time dependent coefficient  $a(t)$  is noisy.  Indeed,  we assume that $a(t)$ is noisy by observed random data $\overline a(t)$ which satisfy  that
\begin{equation}
\overline a(t)= a(t)+ \ep \overline  {\xi}(t)  \label{a}
\end{equation}
where $\overline  {\xi}(t)$ is Brownian motion.    If we  have used a Fourier truncation solution for  \eqref{problem2}, then the regularized solution  would contain some terms such as $\exp\big(p^{2\beta}\int_t^T \int_t^s \overline a(\tau) d\tau ds\big)$, which would lead to  some  complex computations. Hence, we don't follow the truncation method as in \cite{Tuan2}, instead we develop a new method to find a regularized solution. We will apply a new quasi-reversibility method for solving the problem. Further details of this method can be found in  Tuan \cite{Tuan2}. In this case our {\bf main results are Theorems \ref{thm-time-dependent-1} and \ref{thm-high-sob-estimate}.}
}

In this paper, we only study the upper bound of the convergence rate.  In a future work, we will study the minimax rate of  convergence for finding the optimal rate.  The problem of finding minimax rate is a very difficult and interesting problem.

\section{ Regularized solutions for backward problem for nonlinear fractional space diffusion}\label{section2}

\subsection{Some Notation}
We first  introduce notation, and then  state the first set of our  main results in this paper.
We define fractional powers of the Neummann-Laplacian.
{%\color{red}
\begin{equation}
Af:= -\Delta f.
\end{equation}
Since $A$ is a linear densely defined self-adjoint and positive definite
elliptic operator on the connected bounded domain  $\Omega $ with
Neumann  boundary condition, the eigenvalues of $A$ satisfy
\[
\lambda_0=0 < \lambda_1 \le \lambda_2 \le \lambda_3 \le \cdots \le \lambda_p\le \cdots
\]
with  $\lambda_p=p^2  \to \infty $ as $p \to \infty$; see \cite{evan}.
The corresponding eigenfunctions are denoted respectively by
$\varphi_{p}(x)=\sqrt{\frac{2}{\pi}}\cos (px) $.
Thus the eigenpairs $(\lambda_p,\phi_p)$,
$p=0,1,2,...$, satisfy
\[
\begin{cases}
A \varphi_{p}(x)
=
-\lambda_p \phi_{p}(x),
\quad & x \in \Omega \\
\partial_x\phi_{p}(x)
=
0,
\quad & x\in \partial \Omega.
\end{cases}
\]
The functions $\varphi_p$ are normalized so that
$\{\phi_{p}\}_{p=0}^\infty$ is an orthonormal basis of $L^2(\Omega)$.\\
Defining
\[
H^{\gamma}(\Omega)
=
\Bigg\{ v \in L^{2}(\Omega) : \sum\limits_{p=0}^{\infty}
\lambda_{p}^{2\gamma} |\inpro{v}{\phi_{p}}|^{2} <  + \infty \Bigg\},
\]
where $\inpro{\cdot}{\cdot}$ is the inner product in $L^{2}(\Omega)$, then
$H^{\gamma}(\Omega)$ is a Hilbert space equipped with norm
\[
\|v\|_{H^{\gamma}(\Omega)}
=
\left(\sum\limits_{p=1}^{\infty}
\lambda_{p}^{2\gamma} |\inpro{v}{\phi_{p}}|^{2}\right)^{1/2}.
\]
Next we define the   {\bf fractional Laplacian} operator using the spectral theory.
\begin{definition} \label{Def1}
Let $f \in L^2(\Omega)$. For each $\beta >0$, we define fractional Laplacian using the spectral theorem as follows
\begin{equation}
A^\beta f:= (-\Delta )^{\beta } f = \sum\limits_{p=1}^{\infty}p^{2\beta }  \Big<f, \phi_p \Big> \phi_p(x),
\end{equation}
where $\phi_p(x)=\sqrt{\frac{2}{\pi}} \cos(px)$. More details on this fractional Laplacian can be found in \cite{Koba}.
\end{definition}
}
In this section we assume that $a(t)=1$ and $\beta>1/2$.

\subsection{The solution of the problem \eqref{problem2}}

\begin{lemma}\label{lemma-solution-rep}
	If the problem \eqref{problem2}
	 has solution $u$ then it is given by
	\begin{equation}  \label{equality1}
	u(x,t)= \sum_{p=0}^\infty \Big[ e^{(T-t)p^{2\beta}}  \Big<u_T, \phi_p \Big> -\int_t^T  e^{(s-t)p^{2\beta}} g_p(s)ds-\int_t^T  e^{(s-t)p^{2\beta}} F_p(u)(s)ds \Big] \phi_p(x)
	\end{equation}
	where  $g_p(t)= \Big<g(\cdot,t), \phi_p \Big>$ and $F_p(u)(t)= \Big<F((u(\cdot,t) )), \phi_p \Big>$
\end{lemma}

\begin{proof}
	Suppose the Problem \eqref{problem2} has the solution $u$ which given by Fourier series
	\begin{equation}
	u(x,t)= \sum_{p=1}^\infty u_p(t) \phi_p(x), \quad \text{where}\quad u_p(t)= \Big<u(\cdot,t), \phi_p \Big>.
	\end{equation}
Multiplying both sides of the equation $u_t + (-\Delta )^\beta u = F(u(x,t))+g(x,t)$ by $\phi_p(x)$ and integrating over $\Omega$ leads to
\begin{equation}
\frac{d}{dt}u_p(t)+ p^{2\beta} u_p(t)= F_p(u(t))+g_p(t).  \label{eq1}
\end{equation}
Here  we have used Definition \eqref{Def1}.
Multiplying both sides of \eqref{eq1} by $e^{p^{2\beta t}}$, and  by taking the integral from $t$ to $T$ we get
\begin{equation}
\int_t^T \Big( e^{sp^{2\beta }}u_p(s) \Big)'(s)ds= \int_t^T e^{sp^{2\beta }} g_p(s)ds+ \int_t^T e^{sp^{2\beta }} F_p(u(s))ds
\end{equation}
The latter equality can be transformed into
\begin{equation}
u_p(t)= e^{(T-t)p^{2\beta}}  \Big<u_T, \phi_p \Big> -\int_t^T  e^{(s-t)p^{2\beta}} g_p(s)ds-\int_t^T  e^{(s-t)p^{2\beta}} F_p(u)(s)ds
\end{equation}
where we note that $u_p(T)= \Big<u_T, \phi_p \Big> $. This completes the proof of Lemma.
\end{proof}
	First, we  state  following Lemmas that will be used in this paper
	\begin{lemma} \label{lemma2.3}
		[Lemma 2.4 in \cite{Tuan1}]\label{cc}
		Let $p, n \in \mathbb{N}$ such that $0\le p \le n-1$. Assume that $u_T$ is piecewise $C^1$ on $[0,\pi]$ .
		Then
		\begin{equation}\label{y1111}
		 \langle u_T, \phi_p \rangle= \begin{cases}
		\text{ } \text{ } \text{ } \dfrac{1}{n}\sum_{k=1}^n u_T(x_k) - \widetilde {G}_{n0}, & \qquad p=0,\\\\
		\text{ } \text{ } \text{ }  \dfrac{\pi}{n} \sum_{k=1}^n u_T(x_k) \phi_p(x_k) -  \widetilde {G}_{np}, & \qquad 1 \le p \le n-1.
		\end{cases}
		\end{equation}
		where
		\begin{equation}\label{useful2}
		\widetilde {G}_{np}= \begin{cases}
		\text{ } \text{ } \text{ } \sqrt{\dfrac{2}{\pi}}\sum_{l=1}^\infty (-1)^ l  \langle u_T, \phi_{2ln} \rangle, & \qquad p=0,\\\\
		\sum_{l=1}^\infty (-1)^l [   \langle u_T, \phi_{p+2ln} \rangle +  \langle u_T, \phi_{-p+2ln}\rangle], & \qquad 1 \le p \le n-1.
		\end{cases}
		\end{equation}
	\end{lemma}	
Applying Lemma \ref{cc} we obtain the next result.
	\begin{lemma}\label{cc2}
		Let $p, n \in \mathbb{N}$ such that $0\le p \le n-1$. Assume that $x\to g(x,t)$ is piecewise $C^1$ on $[0,\pi]$ .
		Then
		\begin{equation}\label{y1111}
		 \langle g\left(\cdot,t\right), \phi_p \rangle= \begin{cases}
		\text{ } \text{ } \text{ } \dfrac{1}{n}\sum_{k=1}^n g(x_k,t) - \widetilde {H}_{n0}(t), & \qquad p=0,\\\\
		\text{ } \text{ } \text{ }  \dfrac{\pi}{n} \sum_{k=1}^n g(x_k,t) \phi_p(x_k) -  \widetilde {H}_{np}(t), & \qquad 1 \le p \le n-1.
		\end{cases}
		\end{equation}
		where
		\begin{equation}\label{useful2}
		\widetilde {H}_{np}(t)= \begin{cases}
		\text{ } \text{ } \text{ } \sqrt{\dfrac{2}{\pi}}\sum_{l=1}^\infty (-1)^ l  \langle g\left(\cdot,t\right), \phi_{2ln} \rangle, & \qquad p=0,\\\\
		\sum_{l=1}^\infty (-1)^l \Bigg[   \langle g\left(\cdot,t\right), \phi_{p+2ln} \rangle +  \langle g\left(\cdot,t\right), \phi_{-p+2ln} \Big\rangle\Bigg], & \qquad 1 \le p \le n-1.
		\end{cases}
		\end{equation}
	\end{lemma}
	
	%\begin{proof}
%		The proof of Lemma 2.4 can be found in \cite{Tuan1}.
%		\end{proof}

We use the following representation of $u$ in the next lemma  to find  an  estimator of $u(x,t)$.
\begin{lemma}
		Suppose that  problem \eqref{problem2}
		has solution $u$, then $u$ can be  represented  as follows
	\begin{align}
	u(x,t)&= \Phi_{M_n,n}( u_T)(x,t)-\widetilde \Phi_{M_n,n}( g)(x,t)\nonumber\\
	&	-\sum_{p=0}^{M_n}  e^{(T-t)p^{2\beta}}  \widetilde {G}_{np} \phi_p(x)+\sum_{p=0}^{M_n} \Bigg[  \int_t^T  e^{(s-t)p^{2\beta}}  \widetilde {H}_{np}(s)ds\Bigg] \phi_p(x)\nn\\	&-\sum_{p=0}^{M_n}   \Bigg[  \int_t^T e^{(s-t)p^{2\beta}}  F_p(u)(s)ds \Bigg]\phi_p(x)\nonumber\\
	&+\sum_{p=M_n+1}^\infty \Bigg[ e^{(T-t)p^{2\beta}}u_p (T)-\int_t^T e^{(s-t)p^{2\beta}}  g_p(s)ds -\int_t^T e^{(s-t)p^{2\beta}}  F_p(u)(s)ds \Bigg]\phi_p(x).
	\end{align}
	Here $M_n$ is the parameter depending on $n$ such that $0<M_n <n$. The terms $\widetilde H, \widetilde G$ are defined  in Lemma \ref{cc2}.
	$\Phi, \widetilde \Phi$ are defined for all $f \in L^2(\Omega)$ as follows
	{\begin{align} \label{phi1}
{\bf	\Phi}_{M_n,n}(f)(x,t)= \dfrac{1}{n}\sum_{k=1}^n f(x_k, t)+\sum_{p=1}^{M_n}  e^{(T-t)p^{2\beta}} \Bigg[ \dfrac{\pi}{n} \sum_{k=1}^n f(x_k,t) \phi_p(x_k)  \Bigg]\phi_p(x)
	\end{align}
}
	and
	\begin{align} \label{phi2}
	\widetilde \Phi_{M_n,n}( f)(x,t)= \dfrac{1}{n}\sum_{k=1}^n f(x_k,t)+\sum_{p=1}^{M_n} \Bigg[  \int_t^T  e^{(s-t)p^{2\beta}} \left( \dfrac{\pi}{n} \sum_{k=1}^n f(x_k,t) \phi_p(x_k) \right) ds\Bigg] \phi_p(x).
	\end{align}
\end{lemma}

\begin{proof}
	By  Lemma \ref{lemma-solution-rep}, we get
	\begin{align}
	u(x,t)&= \sum_{p=0}^\infty \Bigg[ e^{(T-t)p^{2\beta}}u_p (T)-\int_t^T e^{(s-t)p^{2\beta}}  g_p(s)ds -\int_t^T e^{(s-t)p^{2\beta}}  F_p(u)(s)ds \Bigg]\phi_p(x)\nonumber\\
	&=\underbrace{\sum_{p=0}^{M_n} \Bigg[ e^{(T-t)p^{2\beta}}u_p (T)-\int_t^T e^{(s-t)p^{2\beta}}  g_p(s)ds -\int_t^T e^{(s-t)p^{2\beta}}  F_p(u)(s)ds \Bigg]\phi_p(x)}_{:=A_1}\nonumber\\
	&+\sum_{p=M_n+1}^\infty \Bigg[ e^{(T-t)p^{2\beta}}u_p (T)-\int_t^T e^{(s-t)p^{2\beta}}  g_p(s)ds -\int_t^T e^{(s-t)p^{2\beta}}  F_p(u)(s)ds \Bigg]\phi_p(x). \label{equ1}
	\end{align}
	By using Lemma \ref{cc} and \ref{cc2}, we obtain
	
	\begin{align}
	A_1&=  \dfrac{1}{n}\sum_{k=1}^n u_T(x_k) - \widetilde {G}_{n0}+\sum_{p=1}^{M_n}  e^{(T-t)p^{2\beta}} \Bigg[ \dfrac{\pi}{n} \sum_{k=1}^n u_T(x_k) \phi_p(x_k) -  \widetilde {G}_{np} \Bigg]\phi_p(x)\nonumber\\
	& - \int_t^T \Big[ \dfrac{1}{n}\sum_{k=1}^n g(x_k,s) - \widetilde {H}_{n0}(s) \Big]ds- \sum_{p=1}^{M_n} \Bigg[  \int_t^T  e^{(s-t)p^{2\beta}} \left( \dfrac{\pi}{n} \sum_{k=1}^n g(x_k,t) \phi_p(x_k) -  \widetilde {H}_{np}(s)\right)ds \Bigg] \phi_p(x)\nonumber\\
	&-\sum_{p=0}^{M_n}   \Bigg[  \int_t^T e^{(s-t)p^{2\beta}}  F_p(u)(s)ds \Bigg]\phi_p(x)
	\end{align}
	By a simple computation, the term $A_1$ is equal to
	\begin{align}
	A_1&= \Phi_{M_n,n}( u_T)(x,t)-\widetilde \Phi_{M_n,n}( g)(x,t) -\sum_{p=0}^{M_n}  e^{(T-t)p^{2\beta}}  \widetilde {G}_{np} \phi_p(x)+\sum_{p=0}^{M_n} \Bigg[  \int_t^T  e^{(s-t)p^{2\beta}}  \widetilde {H}_{np}(s)ds\Bigg] \phi_p(x)\nn\\
	& -\sum_{p=0}^{M_n}   \Bigg[  \int_t^T e^{(s-t)p^{2\beta}}  F_p(u)(s)ds \Bigg]\phi_p(x). \label{equ2}
	\end{align}
	Combining \eqref{equ1} and \eqref{equ2} gives the proof of the  Lemma.

\end{proof}

\subsection{ Fourier method and regularization}
	We make use of the  following assumptions on the functions $F, g$
	\begin{enumerate}[{\bf \upshape(i)}]

		\item   $F \in L^\infty (\mathbb{R})$ and  $F$ is a Lipschitz function, i.e. there exists a positive constant K such that
		\begin{equation}\label{lipschitz-condition}
		|F(\xi_1)- F(\xi_2)| \le K |\xi_1-\xi_2|,\quad \forall \xi_1, \xi_2 \in \mathbb{R}.
		\end{equation}
		
		\item There exists positive constant $\gamma>1$  such that
		\begin{equation}  \label{assu2}
		\sup_{0 \le t \le T } \Big[\sum_{p=0}^\infty p^{2\gamma} \Big<g(\cdot,t), \phi_p\Big>^2\Big]  \le \widetilde E_2.
		\end{equation}
		
		\item The  regularized parameter $M_n$ satisfies
		\begin{equation}
		\lim_{n\to +\infty}  \frac{(M_n+1)e^{2T M_n^{2\beta}}}{n}\quad \text{bounded}.
		\end{equation}
		
	\end{enumerate}

\begin{theorem}\label{main-theorem-1}
Suppose $\beta>1/2$ and $a(t)=1$ in equation \eqref{problem2}.
We construct a  regularized  approximate solution of equation \eqref{problem2} denoted by $\overline U_{M_n,n}$  that is defined by the following nonlinear integral equation
\begin{align}
\overline U_{M_n,n}(x,t)&=  {\bf\Phi}_{M_n,n}(\widetilde u_T )(x,t)-\widetilde \Phi_{M_n,n}( \widetilde g)(x,t)-\sum_{p=0}^{M_n}   \Bigg[   \int_t^T e^{(s-t)p^{2\beta}}  F_p(\overline U_{M_n,n})(s)ds \Bigg]\phi_p(x). \label{resol0}
\end{align}
The terms ${\bf\Phi}_{M_n,n}(\widetilde u_T )(x,t)$ and $\widetilde \Phi_{M_n,n}( \widetilde g)(x,t)$ are defined above in equations  \eqref{phi1} and \eqref{phi2} respectively.
Assume that  problem \eqref{problem2} has unique solution  $u \in C([0,T];L^2(\Omega))$. \\
If there exists $\widetilde P_1$ such that
\begin{equation} \label{ass1}
\sup_{0\le t \le T} \sum_{p=1}^\infty e^{2p^{2\beta }t} |<u(\cdot, t),\phi_p>|^2 \le  \widetilde P_1
\end{equation}
Then we have
\begin{align}
\mathbb{\bf E} \Big\|\overline U_{M_n,n}(\cdot,t)-u(\cdot,t)\Big\|^2_{L^2(\Omega)} \le 6e^{-2M_n^{2\beta} t} \Big[  C_3 \frac{(M_n+1) e^{2TM_n^{2\beta}} }{n}+ \widetilde P_1\Big] e^{6K(T-t)}. \label{estimate1}
\end{align}
If there exists $\alpha>0$ and $\widetilde P_2$ such that
\begin{equation}\label{ass2}
\sup_{0\le t \le T} \sum_{p=1}^\infty p^{2\beta \alpha} e^{2p^{2\beta }t} |<u(\cdot, t),\phi_p>|^2 \le  \widetilde P_2
\end{equation}
Then we have
\begin{align}
\mathbb{\bf E} \|\overline U_{M_n,n}(\cdot,t)-u(\cdot,t)\|^2_{L^2(\Omega)}  \le 6e^{-2M_n^{2\beta} t} \Big[  C_3 \frac{(M_n+1) e^{2TM_n^{2\beta}} }{n}+M_n^{-2\beta \alpha} \widetilde P_2\Big] e^{6K(T-t)}. \label{estimate2}
\end{align}
\end{theorem}

\begin{remark}
	In the previous  Theorem, with the estimate in  \eqref{estimate1}, we could get the error estimate  for $t>0$ but the  error estimate for $t=0$ is not useful.  Hence, we need  to assume \eqref{estimate2} to obtain the error estimate for  $t=0$. It is easy to see that  for  $t=0$, the error  is of order
	$\max\left( \frac{(M_n+1) e^{2TM_n^{2\beta}} }{n},M_n^{-2\beta \alpha}  \right)  $.
	
\end{remark}

\begin{remark}
	Let us choose $M_n$ as follows
	\begin{align}
	e^{2T M_n^{2\beta} }= n^{\sigma}, \quad 0<\sigma<1.
	\end{align}
	Then we have
	\begin{equation}
	M_n:= \left( \frac{\sigma}{2T} \log(n)  \right)^{\frac{1}{2\beta}}.
	\end{equation}

	If \eqref{ass1} holds then the error $\mathbb{\bf E} \|\overline U_{M_n,n}(\cdot,t)-u(\cdot,t)\|^2_{L^2(\Omega)} $ is of order $n^{ -\frac{\sigma t}{T} }$.
	
	If \eqref{ass2} holds then the error $\mathbb{\bf E} \|\overline U_{M_n,n}(\cdot,t)-u(\cdot,t)\|^2_{L^2(\Omega)} $ is of order
	 $$n^{ -\frac{\sigma t}{T} }  \max\Big( \frac{\left( \frac{\sigma}{2T} \log(n)  \right)^{\frac{1}{2\beta}}}{n^{1-\sigma}} , \left( \frac{\sigma}{2T} \log(n)  \right)^{-\alpha}     \Big).
	$$
	Hence if   \eqref{ass2} holds and $t=0$ then the error $\mathbb{\bf E} \|\overline U_{M_n,n}(\cdot,0)-u(\cdot,0)\|^2_{L^2(\Omega)} $ is of order
	 $$  \max\Big( \frac{\left( \frac{\sigma}{2T} \log(n)  \right)^{\frac{1}{2\beta}}}{n^{1-\sigma}} , \left( \frac{\sigma}{2T} \log(n)  \right)^{-\alpha}     \Big).$$
\end{remark}

\begin{proof}[\bf Proof of Theorem \ref{main-theorem-1}]
	We divide the proof into two parts. \\
{\bf Part 1.}	
The existence and  uniqueness of the solution to the nonlinear intergral equation \eqref{resol0}. \\
 Let us put
\begin{align}
\mathbb {G} \left(w(x,t)\right)&=  {\bf\Phi}_{M_n,n}(\widetilde u_T )(x,t)-\widetilde \Phi_{M_n,n}( \widetilde g)(x,t)-\sum_{p=0}^{M_n}   \Bigg[   \int_t^T e^{(s-t)p^{2\beta}}  F_p(w)(s)ds \Bigg]\phi_p(x). \label{resol}
\end{align}
for $w \in C([0,T];L^2(\Omega))$.
We claim that for every $v, w \in C([0,T];L^2(\Omega))$
\begin{equation}
\|\mathbb  G^m (v) - \mathbb G^m (w)\| \le \sqrt{ \frac{K^2 T e^{2TM_n^{2\beta} C} }{m!} }   |||v-w|||, \label{fixederror}
\end{equation}
where $|||. ||| $ is the sup norm in $C([0,T];L^2(\Omega))$.
For $m=1$, using H\"{o}lder's inequality we have
\begin{align}
\|\mathbb G(v)(.,t)-\mathbb G(w)(.,t)\|_{L^2(\Omega)} &=\frac{\pi}{2} \sum_{p=0}^{M_n} \Bigg[   \int_t^T e^{(s-t)p^{2\beta}}  \Big(F_p(v)(s) - F_p(w)(s) \Big)ds \Bigg]^2 \nn\\
&\le \frac{\pi}{2} \sum_{p=0}^{M_n} \Bigg[   \int_t^T e^{2(s-t)p^{2\beta}} ds  \int_t^T \Big(F_p(v)(s) - F_p(w)(s) \Big)^2ds \Bigg] \nn\\
&\le \frac{\pi}{2} \sum_{p=0}^{M_n} \Bigg[   e^{2T M_n^{2\beta}} (T-t)  \int_t^T \Big(F_p(v)(s) - F_p(w)(s) \Big)^2ds \Bigg] \nn\\
&= e^{2T M_n^{2\beta}} (T-t) \sum_{p=0}^{M_n} \Bigg[     \int_t^T \Big(F_p(v)(s) - F_p(w)(s) \Big)^2ds \Bigg] \nn\\
&\le e^{2T M_n^{2\beta}} (T-t) \|F(v)(.,t)-F(w)(.,t)\|^2_{L^2(\Omega)}. \label{err1}
\end{align}
Using the fact that  $F$ is globally Lipschitz,  we get
\begin{align}
\|F(v)(.,t)-F(w)(.,t)\|^2_{L^2(\Omega)} \le K^2 \|v(.,t)-w(.,t)\|2_{L^2(\Omega)} \le C K^2  |||v-w|||. \label{err2}
\end{align}
Combining \eqref{err1} and \eqref{err2}, we conclude that \eqref{fixederror} holds for $m=1$. By a similar method  as above, we can show that \eqref{fixederror} holds for $m=j$ for $j \in \mathbb{N}$.  It is obvious that
\begin{equation}
\lim_{m \to +\infty} \sqrt{ \frac{K^2 T e^{2TM_n^{2\beta} C} }{m!} } =0.
\end{equation}
It implies that there exists a positive integer number $m_0$, such that $\mathbb {G}^{m_0}$ is a contraction. It follows that the equation $\mathbb {G}^{m_0} v=v  $ has unique solution $\overline U_{M_n,n} \in C([0,T];L^2(\Omega))$. We claim that $ \mathbb{G} \left(\overline U_{M_n,n} \right) = \overline U_{M_n,n}$. In fact, we have
\begin{equation}
\mathbb{G} \left( \mathbb{G}^{m_0} \left(\overline U_{M_n,n} \right) \right)= \mathbb{G}    \left(\overline U_{M_n,n} \right)
\end{equation}
Hence
\begin{equation}
\mathbb{G}^{m_0} \left( \mathbb{G} \left(\overline U_{M_n,n} \right) \right)= \mathbb{G}    \left(\overline U_{M_n,n} \right)
\end{equation}
The latter equality leads to $\mathbb{G} \left(\overline U_{M_n,n} \right)$ is a fixed point of $\mathbb{G}^{m_0} $.
By  the uniqueness of the fixed point of $\mathbb{G}^{m_0} $ , we can conclude that $ \mathbb{G} \left(\overline U_{M_n,n} \right) = \overline U_{M_n,n}$. Part 1 is completely proved.

{\bf Part 2.} The error estimate between the regularized solution
$\overline U_{M_n,n}$ and the exact solution $u$.

From Lemma \ref{cc}  we get
\begin{align}
& \overline U_{M_n,n}(x,t)-u(x,t)\nonumber\\
&\quad \quad \quad  =\underbrace{\Phi_{M_n,n}( \widetilde u_T)(x,t) - \Phi_{M_n,n}( u_T)(x,t)}_{B_{1,M,n}(x,t) }+\underbrace{\widetilde \Phi_{M_n,n}( \widetilde g)(x,t)-\widetilde \Phi_{M_n,n}( g)(x,t)}_{{B_{2,M,n}(x,t) }}\nonumber\\
&\quad \quad \quad +\underbrace{\sum_{p=0}^{M_n}  e^{(T-t)p^{2\beta}}  \widetilde {G}_{np} \phi_p(x)}_{B_{3,M,n}(x,t)}-\underbrace{\sum_{p=0}^{M_n} \Bigg[  \int_t^T  e^{(s-t)p^{2\beta}}  \widetilde {H}_{np}(s)ds\Bigg] \phi_p(x)}_{{B_{4,M,n}(x,t)}}\nonumber\\
&\quad \quad \quad -\underbrace{\sum_{p=0}^{M_n}   \Bigg[   \int_t^T e^{(s-t)p^{2\beta}}  \Big( F_p(\overline U_{M_n,n})-F_p(u)(s)\Big)ds \Bigg]\phi_p(x)}_{{B_{5,M,n}(x,t)}}\nonumber\\
&\quad \quad \quad -\underbrace{\sum_{p=M_n+1}^\infty \Bigg[ e^{(T-t)p^{2\beta}}u_p (T)-\int_t^T e^{(s-t)p^{2\beta}}  g_p(s)ds -\int_t^T e^{(s-t)p^{2\beta}}  F_p(u)(s)ds \Bigg]\phi_p(x)}_{{B_{6,M_n,n}(x,t)}}
\end{align}
This implies that
\begin{align}
\|\overline U_{M_n,n}(\cdot,t)-u(\cdot,t)\|^2_{L^2(\Omega)} &\le 6 \|B_{1,M,n}\|^2_{L^2(\Omega)}+6 \|B_{2,M,n}\|^2_{L^2(\Omega)}+6 \|B_{3,M,n}\|^2_{L^2(\Omega)}\nonumber\\
&+6 \|B_{4,M,n}\|^2_{L^2(\Omega)}+6 \|B_{5,M,n}\|^2_{L^2(\Omega)}+6 \|B_{6,M,n}\|^2_{L^2(\Omega)}.
\end{align}
{\bf Step 1. } Estimate $\mathbb{\bf E} \|B_{1,M_n,n}\|^2_{L^2(\Omega)}.$ \\
Using the fact that $\widetilde u_T(x_k)=u_T(x_k)+\sigma_k \epsilon_k$, we get
\begin{align}
&\Phi_{M_n,n}( \widetilde u_T)(x,t) - \Phi_{M_n,n}( u_T)(x,t)\nonumber\\
&\quad \quad \quad \quad = \dfrac{1}{n} \sum_{k=1}^n  \Big[\widetilde u_T(x_k)(x_k)-u_T(x_k)\Big]\nonumber\\
&\quad \quad \quad \quad+\sum_{p=1}^{M_n}  e^{(T-t)p^{2\beta}} \Bigg[ \dfrac{\pi}{n} \sum_{k=1}^n \Big(u_T^\epsilon(x_k)-u_T(x_k)\Big) \phi_p(x_k)  \Bigg]\phi_p(x)\nonumber\\
&\quad \quad \quad \quad=\dfrac{1}{n} \sum_{k=1}^n  \sigma_k \epsilon_k+\sum_{p=1}^{M_n}  e^{(T-t)p^{2\beta}} \Bigg[ \dfrac{\pi}{n} \sum_{k=1}^n \sigma_k \epsilon_k\phi_p(x_k)  \Bigg]\phi_p(x).
\end{align}
The Parseval's identity implies that
\begin{align}
\|B_{1,M_n,n}\|^2_{L^2(\Omega)}&= \Big\|\Phi_{M,n}( \widetilde u_T(\cdot,t) - \Phi_{M,n}( u_T)(\cdot,t)\Big\|^2\nonumber\\
&=\frac{1}{n^2}  \Big[\sum_{k=1}^n  \sigma_k \epsilon_k \Big]^2+\sum_{p=1}^{M_n}  e^{2(T-t)p^{2\beta}} \Bigg[ \dfrac{\pi}{n} \sum_{k=1}^n \sigma_k \epsilon_k\phi_p(x_k)  \Bigg]^2.
\end{align}
Since the noises $ \epsilon_{k}$ are mutually independent, we obtain $ \mathbb{\bf E} ( \epsilon_i \epsilon_k) = 0$. Hence
\begin{align}
\mathbb{\bf E} \|B_{1,M_n,n}\|^2_{L^2(\Omega)}&=\frac{1}{n^2} \sum_{k=1}^{n} \sigma_k^2\mathbb{\bf E} \epsilon_k^2+\sum_{p=1}^{M_n}  e^{2(T-t)p^{2\beta}} \frac{\pi^2}{n^2} \sum_{k=1}^{n} \sigma_k^2 \mathbb{\bf E} \epsilon_k^2\nonumber\\
&\le \frac{V_{\max}^2}{n}+ \frac{\pi^2 V_{\max}^2}{n} \sum_{p=1}^{M_n}  e^{2(T-t)p^{2\beta}}\le \pi^2 V_{\max}^2 \frac{ \sum_{p=0}^{M_n}  e^{2(T-t)p^{2\beta}}}{n}.
\end{align}
{\bf Step 2. } Estimate $\mathbb{\bf E} \|B_{2,M_n,n}\|^2_{L^2(\Omega)}.$ \\
From equations \eqref{noise}, \eqref{phi1}, \eqref{phi2}, we deduce that
\begin{align}
&\widetilde \Phi_{M_n,n}( \widetilde g)(\cdot,t)-\widetilde \Phi_{M_n,n}( g)(\cdot,t)\nonumber\\
&\quad \quad = \dfrac{1}{n}\sum_{k=1}^n \vartheta\xi_{k}(t)+\sum_{p=1}^{M_n} \Bigg[  \int_t^T  e^{(s-t)p^{2\beta}} \left( \dfrac{\pi}{n} \sum_{k=1}^n \vartheta\xi_{k}(s) \phi_p(x_k) \right) ds\Bigg] \phi_p(x).
\end{align}
The Parseval's identity implies that
\begin{align}
\Big\|\widetilde \Phi_{M_n,n}( \widetilde g)(\cdot,t)-\widetilde \Phi_{M_n,n}( g)(\cdot,t)\Big\|_{L^2(\Omega)}^2&= \frac{1}{n^2} \Big[ \sum_{k=1}^n \vartheta\xi_{k}(t)\Big]^2\nonumber\\
&+\sum_{p=1}^{M_n} \Bigg[  \int_t^T  e^{(s-t)p^{2\beta}} \left( \dfrac{\pi}{n} \sum_{k=1}^n \vartheta\xi_{k}(s) \phi_p(x_k) \right) ds\Bigg]^2
\end{align}
From the properties of Brownian motion, we know that $\mathbb{\bf E}[\xi_{i}(t)\xi_{k}(t)] = 0$ for $k \ne i$ and $\mathbb{\bf E}\xi_{k}^2(t) = t$. By the H\"{o}lder's inequality, we obtain
\begin{align}
\mathbb{\bf E} \|B_{2,M_n,n}\|^2_{L^2(\Omega)} &\le \frac{\vartheta^2 t}{n}+\sum_{p=1}^{M_n} \Big[ \int_t^T e^{2(s-t)p^{2\beta}}ds  \int_t^T \frac{\pi^2}{n^2} \sum_{k=1}^n  \vartheta^2 \mathbb{\bf E}\xi_{k}^2(s)\phi_p^2(x_k) ds \Big]\nonumber\\
&\le \frac{\vartheta^2 T}{n}+\vartheta^2 T^3\frac{\sum_{p=1}^{M_n}  e^{2(T-t)p^{2\beta}}}{n} \le \vartheta^2 (T+T^3)\frac{\sum_{p=0}^{M_n}  e^{2(T-t)p^{2\beta}}}{n} .
\end{align}
{\bf Step 3. } Estimate $ \|B_{3,M_n,n}\|^2_{L^2(\Omega)}.$ \\
First,  we have the following inequality
\begin{align}
|F_p(u(t))| \le \|F(u(\cdot,t))\|_{L^2(\Omega)} &\le \|F(0)\|_{L^2(\Omega)}+ K\|u\|_{L^2(\Omega)}\nonumber\\
&\le \underbrace{\|F(0)\|_{L^2(\Omega)}+ \max(K,1) \sup_{0 \le t \le T }\|u(.,t)\|_{L^2(\Omega)}}_{:=E},
\end{align}
and $\|u(.,0)\|_{L^2(\Omega)} \le E$.
This implies that for all $p \in \mathbb{N}, p \ge 1$
\begin{align}
\left|\Big<  u(\cdot,T), \phi_p\Big> \right|&=\left| e^{-Tp^{2\beta}} u_p (0)+ \int_0^T  e^{(s-T)p^{2\beta}} F_p(u)(s)ds \right| \nonumber\\
&\le e^{-Tp^{2\beta}} \|u(.,0)\|_{L^2(\Omega)}+\frac{1-e^{-Tp^{2\beta}}}{p^{2\beta}} \|F(u)\|_{L^\infty} \le \frac{E}{T p^{2\beta}}+ \frac{E}{p^{2\beta}} \le \frac{\widetilde E_1}{p^{2\beta}}. \label{iq1}
\end{align}
where $\widetilde E_1= \frac{E}{T}+E$. \\
Now, we estimate the term $\widetilde G_{np} $ for $p \ge 0$. From equations  \eqref{useful2} and \eqref{iq1}, we have the bound for $\widetilde G_{n0}$
\begin{equation}
\widetilde G_{n0} \le \sqrt{\dfrac{2}{\pi}}\sum_{l=1}^\infty \left| {\langle u_T,{\phi _{2ln}}\rangle } \right|
\le \sqrt {\frac{2}{\pi }} \sum\limits_{l = 1}^\infty  \frac{ \widetilde E_1 }{2^{2\beta}l^{2\beta} n^{2\beta}} .\label{iq22}
\end{equation}
and the bound for $\widetilde G_{np}$
\begin{align}
\widetilde G_{np} &\le \sum_{l=1}^\infty  \Bigg|  \Big \langle u_T, \phi_{p+2ln} \Big\rangle + \Big \langle u_T, \phi_{-p+2ln} \Big\rangle\Bigg|\nonumber\\
&\le  \widetilde E \Bigg[\sum_{l=1}^\infty \frac{1}{(p+2ln)^{2\beta}} +\sum_{l=1}^\infty \frac{1}{(-p+2ln)^{2\beta}}\Bigg] \le \frac{2\widetilde E}{4^\beta} \Big( \sum\limits_{l = 1}^\infty  \frac{1}{l^{2\beta}} \Big) \frac{1}{n^{2\beta}}. \label{iq33}
\end{align}	
Since $\beta >\frac{1}{2}$ by assumption, we know that the series   $\sum\limits_{l = 1}^\infty  \frac{1}{l^{2\beta}}$ converges. Let us denote
\begin{equation}\label{beta-condition}
C_1(\beta, \widetilde E_1)=2\sqrt {\frac{2}{\pi }} \frac{\widetilde E}{4^\beta}  \sum\limits_{l = 1}^\infty  \frac{1}{l^{2\beta}}.
\end{equation}
Then combining \eqref{iq22} and \eqref{iq33} gives
\begin{equation}
\widetilde G_{np} \le \frac{ C_1(\beta, \widetilde E_1)}{n^{2\beta}},\quad \text{for all }\quad p \ge 0.  \label{Gnp}
\end{equation}
This leads to  the following estimate
\begin{align}
\|B_{3,M_n,n}\|^2_{L^2(\Omega)}= \sum_{p=0}^{M_n}  e^{2(T-t)p^{2\beta}} \Big| \widetilde {G}_{np}\Big|^2  \le C_1^2(\beta, \widetilde E_1) \frac{\sum_{p=0}^{M_n}  e^{2(T-t)p^{2\beta}}}{n^{4\beta}} .
\end{align}
{\bf Step 4. } Estimate $ \|B_{4,M_n,n}\|^2_{L^2(\Omega)}.$ \\
Using  the assumption \eqref{assu2},
we have
\begin{align} \label{iq4}
\Big<g(\cdot,t), \phi_p\Big>=p^{-\gamma} p^{\gamma} <g(\cdot,t), \phi_p> \le \frac{ \widetilde E_2}{p^{\gamma}}
\end{align}
for $\gamma >1$.
From equation \eqref{iq4}, we estimate  $\widetilde H_{n0}(t)$ as follows
\begin{equation}
\widetilde H_{n0}(t) \le \sqrt{\dfrac{2}{\pi}}\sum_{l=1}^\infty \left| {\langle g\left( \cdot,t \right),{\phi _{2ln}}\rangle } \right|
\le \sqrt {\frac{2}{\pi }} \sum\limits_{l = 1}^\infty  \frac{ \widetilde E_2}{2^{\gamma}l^{\gamma} n^{\gamma}}  \le \frac{C_2(\gamma, \widetilde E_2)}{n^{\gamma}} .
\end{equation}
For $p \ge 1$
\begin{align}
\widetilde H_{np}(t) &\le \sum_{l=1}^\infty  \Bigg|  \Big \langle g\left(\cdot,t\right), \phi_{p+2ln} \Big\rangle + \Big \langle g\left(\cdot,t\right), \phi_{-p+2ln} \Big\rangle\Bigg|\nonumber\\
&\le  \widetilde E_2 \Bigg[\sum_{l=1}^\infty \frac{1}{(p+2ln)^{\gamma}} +\sum_{l=1}^\infty \frac{1}{(-p+2ln)^{\gamma}}\Bigg] \le \frac{2\widetilde E_2}{2^\gamma} \Big( \sum\limits_{l = 1}^\infty  \frac{1}{l^{\gamma}} \Big) \frac{1}{n^{\gamma}} \le \frac{C_2(\gamma, \widetilde E_2)}{n^{\gamma}} .
\end{align}	
where $$C_2(\gamma, \widetilde E_2)=2\sqrt {\frac{2}{\pi }} \frac{\widetilde E_2}{2^\gamma}  \sum\limits_{l = 1}^\infty  \frac{1}{l^{\gamma}}. $$
This leads to the following estimation
\begin{align}
\|B_{4,M_n,n}\|^2_{L^2(\Omega)}&=\sum_{p=0}^{M_n} \Bigg[  \int_t^T  e^{(s-t)p^{2\beta}}  \widetilde {H}_{np}(s)ds\Bigg]^2 \le T^2 C_2^2(\gamma, \widetilde E_2) \frac{ \sum_{p=0}^{M_n}  e^{2(T-t)p^{2\beta}}}{n^{2\gamma}}.
\end{align}
{\bf Step 5.} Estimate $\mathbb{\bf E} \|B_{5,M_n,n}\|^2_{L^2(\Omega)}$.
Using   H\"{o}lder's  inequality and Lipschitz property of $F$, we have
\begin{align}
\|B_{5,M_n,n}\|^2_{L^2(\Omega)}&= \sum_{p=0}^{M_n}   \Bigg[   \int_t^T e^{(s-t)p^{2\beta}}  \Big( F_p(\overline U_{M_n,n})-F_p(u)(s)\Big)ds \Bigg]^2 \nonumber\\
& \le \sum_{p=0}^{M_n} \int_t^T e^{2(s-t)p^{2\beta}} \Big( F_p(\overline U_{M_n,n})-F_p(u)(s)\Big)^2 ds \nonumber\\
&\le \int_t^T e^{2(s-t)M_n^{2\beta}} \left[\sum_{p=0}^{M_n}  \Big( F_p(\overline U_{M_n,n})-F_p(u)(s)\Big)^2 \right]ds \nonumber\\
&\le K \int_t^T e^{2(s-t) M_n^{2\beta}} \|\overline U_{M_n,n}(.,s)-u(.,s)\|^2_{L^2(\Omega)}ds.
\end{align}
This implies that
\begin{equation}
\mathbb{\bf E}\|B_{5,M_n,n}\|^2_{L^2(\Omega)} \le  K \int_t^T e^{2(s-t)M_n^{2\beta}} \mathbb{\bf E} \|\overline U_{M_n,n}(.,s)-u(.,s)\|^2_{L^2(\Omega)} ds.
\end{equation}
{\bf Step 6.} Estimate $\|B_{6,M_n,n}\|^2_{L^2(\Omega)} $.\\
\begin{itemize}
	\item If \eqref{ass1} holds then we get
	\begin{align}
	\|B_{6,M_n,n}\|^2_{L^2(\Omega)} = \sum_{p=M_n+1}^\infty  \Big<u(\cdot,t), \phi_p \Big>^2& = \sum_{p=M_n+1}^\infty  e^{-2p^{2\beta} t}  e^{2p^{2\beta} t}\Big<u(\cdot,t), \phi_p \Big>^2\nn\\
	&\le e^{-2M_n^{2\beta} t} \widetilde P_1.
	\end{align}
	\item If  \eqref{ass2} holds then we get
	\begin{align}
	\|B_{6,M_n,n}\|^2_{L^2(\Omega)}  &= \sum_{p=M_n+1}^\infty  \Big<u(\cdot,t), \phi_p \Big>^2\nonumber\\
	&= \sum_{p=M_n+1}^\infty p^{-2\beta \alpha} e^{-2p^{2\beta} t}  e^{2p^{2\beta} t} p^{2\beta \alpha}\Big<u(\cdot,t), \phi_p \Big>^2\le M_n^{-2\beta \alpha}e^{-2M_n^{2\beta} t} \widetilde P_2.
	\end{align}
\end{itemize}
We divide  proof of the theorem in two cases:\\
{\bf Case 1:   When equation \eqref{ass1} holds}.\\
Combining  six steps  above, we get
\begin{align*}
&\mathbb{\bf E} \|\overline U_{M_n,n}(\cdot,t)-u(\cdot,t)\|^2_{L^2(\Omega)} \nn\\
\quad &\le \quad  6 \mathbb{\bf E} \|B_{1,M_n,n}\|^2+6 \mathbb{\bf E}\|B_{2,M_n,n}\|^2+6 \|B_{3,M_n,n}\|^2\nonumber\\
&+ \quad 6 \|B_{4,M_n,n}\|^2+6 \mathbb{\bf E}\|B_{5,M_n,n}\|^2+6 \|B_{6,M_n,n}\|^2 \nonumber\\
&\le \quad6 \Big( \pi^2 V_{\max}^2+\vartheta^2 (T+T^3)\Big) \frac{ \sum_{p=0}^M  e^{2(T-t)p^{2\beta}}}{n} \nonumber\\
&+\quad 6C_1^2(\beta, \widetilde E_1) \frac{\sum_{p=0}^{M_n}  e^{2(T-t)p^{2\beta}}}{n^{4\beta}}+6T^2 C_2^2(\gamma, \widetilde E_2) \frac{ \sum_{p=0}^{M_n}  e^{2(T-t)p^{2\beta}}}{n^{2\gamma}}\nonumber\\
&+\quad6  e^{-2M_n^{2\beta} t} \widetilde P_1+6K \int_t^T e^{2(s-t)M_n^{2\beta}} \mathbb{\bf E} \|\overline U_{M_n,n}(.,s)-u(.,s)\|^2ds.
\end{align*}
This leads to
\begin{align*}
&\mathbb{\bf E} \|\overline U_{M_n,n}(\cdot,t)-u(\cdot,t)\|^2_{L^2(\Omega)} \nonumber\\
&\le 6\Bigg[ \Big( \pi^2 V_{\max}^2+\vartheta^2 (T+T^3)\Big)+C_1^2(\beta, \widetilde E_1)+T^2 C_2^2(\gamma, \widetilde E_2) \Bigg] \frac{(M_n+1) e^{2(T-t)M_n^{2\beta}} }{\min(n, n^{4\beta}, n^{2\gamma})}\nonumber\\
&+6  e^{-2M_n^{2\beta} t} \widetilde P_1+6K \int_t^T e^{2(s-t)M^{2\beta}} \mathbb{\bf E} \|\overline U_{M_n,n}(.,s)-u(.,s)\|^2_{L^2(\Omega)}ds.
\end{align*}
Multiplying the latter inequality with $e^{2M_n^{2\beta} t}$ , we obtain
\begin{align}
e^{2M_n^{2\beta} t}\mathbb{\bf E} \|\overline U_{M_n,n}(\cdot,t)-u(\cdot,t)\|^2_{L^2(\Omega)} &\le 6 C_3 \frac{(M+1) e^{2TM^{2\beta}} }{\min(n, n^{4\beta}, n^{2\gamma})}+6 \widetilde P_1\nonumber\\
&+ 6K \int_t^T e^{2sM^{2\beta}} \mathbb{\bf E} \|\overline U_{M_n,n}(.,s)-u(.,s)\|^2_{L^2(\Omega)}ds.
\end{align}
Since $6 C_3 \frac{(M+1) e^{2TM^{2\beta}} }{\min(n, n^{4\beta}, n^{2\gamma})}$ does not depend on $t$, using Gronwall's inequality, we obtain
\begin{align}
e^{2M_n^{2\beta} t}\mathbb{\bf E} \|\overline U_{M_n,n}(\cdot,t)-u(\cdot,t)\|^2_{L^2(\Omega)} \le \Big[ 6 C_3 \frac{(M_n+1) e^{2TM_n^{2\beta}} }{n}+6 \widetilde P_1\Big] e^{6K(T-t)}
\end{align}
{\bf Case 2: When equation  \eqref{ass2} holds}.\\
By similar method as in the previous case, we get
\begin{align}
e^{2M_n^{2\beta} t}\mathbb{\bf E} \|\overline U_{M_n,n}(\cdot,t)-u(\cdot,t)\|^2_{L^2(\Omega)} \le \Big[ 6 C_3 \frac{(M_n+1) e^{2TM_n^{2\beta}} }{n}+6 M_n^{-2\beta \alpha}\widetilde P_2\Big] e^{6K(T-t)}
\end{align}
when \eqref{ass2} holds.
\end{proof}

\begin{remark}
	In a future work, we will study the random case for final value problem for the  time and space fractional diffusion equation in the sense of Chen et al.  \cite{Nane}.
\end{remark}
 {
	\begin{remark}
		In  theorem \ref{main-theorem-1} we assumed that  the source function  $F$ is  globally Lispchitz. In some   applications of our model, the extension to  locally Lipschitz source functions is required. 	Suppose that the source function
		$F: \mathbb{R} \to
		\mathbb{R}$ satisfies that
		\begin{equation}\label{dkff}
		\left| {F(u) - F(v)} \right|
		\leqslant K_F(Q)\left| {u - v} \right|,
		\end{equation}
for each $Q>0$ and
		for any $u$, $v$  with $|u|, |v| \leqslant Q$,
		where
		\[
		K_F(Q): =
		\sup \left\{ {\left| {\frac{{F(u) - F(v)}}{{u - v}}}
			\right|:\left| u \right|,\left| v \right|
			\leqslant Q,u \ne v}
		\right\} <  + \infty.
		\]
		Suppose  that $K_F(Q)$ is increasing and   $\mathop {\lim }\limits_{Q \to
			+ \infty } K_F(Q) =  + \infty $.
		In this case, as used by Tuan \cite{Tuan} we
		approximate $F$ by $\overline F_{Q_n}$ defined by
		\[
		\overline F_{Q_n}\left(u(x,t) \right)
		=
		\begin{cases}
		F(Q_n), &\quad u(x,t) >Q_n,\\
		F(u(x,t)), &\quad - Q_n \le u(x,t)\le Q_n, \\
		F(-Q_n),&\quad u(x,t) < - Q_n.
		\end{cases}
		\]
		where  the sequence $Q_n \to +\infty $ as $n \to +\infty$.  Using equation \eqref{resol0}, we introduce  the following  regularized solution
		\begin{align}
		\overline U_{M_n,Q_n,n}(x,t)&=  {\bf\Phi}_{M_n,n}(\widetilde u_T )(x,t)-\widetilde \Phi_{M_n,n}( \widetilde g)(x,t)-\sum_{p=0}^{M_n}   \Bigg[   \int_t^T e^{(s-t)p^{2\beta}}  \overline F_{Q_n,p}(\overline U_{M_n,Q_n,n})(s)ds \Bigg]\phi_p(x). \label{resol2}
		\end{align}
		Using a similar method as in the proof of Theorem \ref{main-theorem-1}, we can get the  error estimate  of $u$ by  $	\overline U_{M_n,Q_n,n}$. We   omit the details of the  proof here.
	\end{remark}
}
\begin{remark} \label{remark2}
	In  Theorem \ref{main-theorem-1}, to obtain the error estimate, we require the strong assumptions   \eqref{ass1} and \eqref{ass2} about $u$. This is a
	limitation of Theorem 1. There are not many functions $u$ that satisfies these conditions. Especially in practice, these
	conditions are more difficult to be satisfied and checked. To remove this limitation, we introduce a new regularization
	solution and introduce a new technique to estimate the error in Theorem \ref{main-theorem-1}. In fact, in the next  theorem we only need a
	weaker assumption for $u $. We assume that  $u \in C([0,T];H^\gamma(0,\pi))$ for any $\gamma>0$. This condition is more  natural.
\end{remark}

We have the following Lemma which gives a new representation of the solution when $g=0$ in equation \eqref{problem2}.
\begin{lemma}\label{lemma-solution-rep2}
	Suppose that $g=0$ and that  the problem \eqref{problem2}
	has solution $u$ then it is represented  as follows
	\begin{align}
		u(x,t)&= \Phi_{M_n,n}( u_T)(x,t)	-\sum_{p=0}^{M_n}  e^{(T-t)p^{2\beta}}  \widetilde {G}_{np} \phi_p(x)\nn\\	&-\sum_{p=0}^{M_n}   \Bigg[  \int_t^T e^{(s-t)p^{2\beta}}  F_p(u)(s)ds \Bigg]\phi_p(x)+\sum_{p=M_n+1}^\infty \Bigg[ \int_0^t e^{(s-t)p^{2\beta}} F_p(u)(s)ds\Bigg]\phi_p(x)\nn\\
		&+\sum_{p=M_n+1}^\infty  e^{-tp^{2\beta}}  u_p(0)\phi_p(x).
	\end{align}
\end{lemma}
\begin{proof}
The proof is a simple adaptation of Lemma \ref{lemma-solution-rep} and we omit it here.
\end{proof}
\begin{theorem}\label{main-theorem-2}
	Let $g=0$.
Assume that $5K T <1$ (where $K$ is the Lipschitz constant of $F$ in equation \eqref{lipschitz-condition}  and the problem \eqref{problem2} has unique solution $u$ such that $u \in C([0,T];H^{\gamma}(\Omega))$.
	We construct another  regularized solution $\widehat U_{M_n,n}$  which is defined by the following nonlinear integral
	\begin{align}
\widehat	U_{M_n,n}(x,t)&=  \Phi_{M_n,n}(\widetilde u_T )(x,t)-\sum_{p=0}^{M_n}   \Bigg[   \int_t^T e^{(s-t)p^{2\beta}}  F_p(\widehat U_{M_n,n})(s)ds \Bigg]\phi_p(x)\nonumber\\
	&+	\sum_{p=M_n+1}^{\infty}   \Bigg[   \int_0^t e^{(s-t)p^{2\beta}}  F_p(\widehat U_{M_n,n})(s)ds \Bigg]\phi_p(x). \label{resol2}
	\end{align}
Moreover, we have the following estimate
\begin{align}
\mathbb{\bf E} \|\widehat U_{M_n,n}(\cdot,t)-u(\cdot,t)\|^2_{L^2(\Omega)}\le e^{-2M_n^{2\beta} t} \frac{ \Big(5\pi^2 V_{\max}^2 +C_1^2(\beta, \widetilde E_1)\Big) \frac{(M_n+1) e^{2T M_n^{2\beta}}}{n}
	+5M_n^{-2\beta \gamma}  \|u(0)\|^2_{H^{\gamma}(\Omega)} }{1-5kT}.
\end{align}
\end{theorem}

\begin{proof}
{\bf Part A.}  The integral equation \eqref{resol2} has unique solution in $ C([0,T];L^2(\Omega)).$
Let us define on $ C([0, T];H) $ the following Bielecki norm
	\begin{equation*}
	\|f\|_1= \sup_{0 \le t \le T} e^{t M_n} \|f(t)\|,~~\text{for all}~~f \in C([0, T];L^2(\Omega)) .
	\end{equation*}
	It is easy to show that $\|.\|_1$ is a norm of $C([0, T];L^2(\Omega)) $.
	For $w \in C([0, T];L^2(\Omega)),$ we consider the following functional
	\begin{align}
\mathcal{	J}(w)(t) &=\Phi_{M_n,n}(\widetilde u_T )(x,t)-\sum_{p=0}^{M_n}   \Bigg[   \int_t^T e^{(s-t)p^{2\beta}}  F_p(w)(s)ds \Bigg]\phi_p(x)\nonumber\\
	&+	\sum_{p=M_n+1}^{\infty}   \Bigg[   \int_0^t e^{(s-t)p^{2\beta}}  F_p(w)(s)ds \Bigg]\phi_p(x), \label{eq55}
	\end{align}
	where $F_p (w) (s) =  \big \langle F(s,w(s)), \phi_p \big \rangle $.  We shall prove that, for every $w_1,w_2 \in C([0, T];L^2(\Omega))$,
	\bes
	\|\mathcal{	J}(w_1)-\mathcal{	J}(w_2)\|_1 \le KT \|w_1-w_2\|_1.  \label{J}
	\ens	
	First, by using  H\"older's inequality and Lipschitz condition of $F$, we have the following estimates for all $t\in [0,T]$
	\begin{align}
&\sum_{p=0}^{M_n}  \Bigg( \int_{t}^{T}e^{(s-t) p^{2\beta} }
	\Big[F_{p}(w_1)(s)- F_p(w_2)(s) \Big] ds \Bigg)^2\nn\\
	& \quad \quad \quad \quad\le (T-t) \sum_{p=0}^{M_n}  \int_{t}^{T} \Big| e^{(s-t) p^{2\beta}  }
	\Big[ F_{p}(w_1)(s)- F_p(w_2)(s) \Big] \Big|^2ds  \nn\\
	& \quad \quad \quad \quad\le (T-t) \sum_{p=0}^{M_n}   \int_t^T  e^{2(s-t)M_n^{2\beta}} \Big|F_{p}(w_1)(s)- F_p(w_2)(s) \Big|^2ds\nn\\
	& \quad \quad \quad \quad\le K^2 (T-t) \int_t^T e^{2(s-t)M_n^{2\beta}} \|w_1(s)-w_2(s)\|^2ds\nn\\
	& \quad \quad \quad \quad \le e^{-2tM_n^{2\beta}}  K^2 (T-t)^2 \sup_{0\le s \le T} e^{2sM_n^{2\beta}} \|w_1(s)-w_2(s)\|^2 \nn\\
	& \quad \quad \quad \quad=e^{-2tM_n^{2\beta}}  K^2 (T-t)^2 \|w_1-w_2\|_1^2, \label{es1}
	\end{align}
	and
	\begin{align}
&	\sum_{p=M_n+1}^{\infty}  \Bigg( \int_{0}^{t} e^{(s-t) M_n^{2\beta} }
	\Big[ F_{p}(w_1)(s)- F_p(w_2)(s) \Big] ds \Bigg)^2\nn\\
	&\quad \quad \quad \quad \le   t  	\sum_{p=M_n+1}^{\infty}  \int_{0}^{t} \Big| e^{(s-t) M_n^{2\beta} }
	\Big[F_{p}(w_1)(s)- F_p(w_2)(s)\Big] \Big|^2ds  \nn\\
	& \quad \quad \quad \quad  \le t \sum_{p=M_n+1}^{\infty}  \int_0^t  e^{2(s-t)M_n^{2\beta} } \big|F_{p}(w_1)(s)- F_p(w_2)(s) \big|^2ds\nn\\
	& \quad \quad \quad \quad\le  K^2 t \int_0^t e^{2(s-t)M_n^{2\beta} }  \|w_1(s)-w_2(s)\|^2ds\nn\\
	& \quad \quad \quad \quad\le   e^{-2tM_n^{2\beta} }  K^2 t^2 \sup_{0\le s \le T} e^{2sM_n^{2\beta} } \|w_1(s)-w_2(s)\|^2 \nn\\
	&\quad \quad \quad \quad =  e^{-2tM_n^{2\beta} }  K^2 t^2 \|w_1-w_2\|_1^2. \label{es2}
	\end{align}
	From the definition of $\mathcal J$ in \eqref{eq55}, we have
	\begin{align}
	\mathcal{J}(w_1)(t)- \mathcal{J}(w_2)(t) & = \sum_{p=0}^{M_n} \Bigg( - \int_{t}^{T}e^{(s-t) M_n^{2\beta}  }
	\Big[F_{p}(w_1)(s)- F_p(w_2)(s) \Big] ds \Bigg) \phi_p(x)\nn\\
	&+ \sum_{p=M_n+1}^{\infty}  \Bigg( \int_{0}^{t} e^{(s-t) M_n^{2\beta}  }
	\Big[ F_{p}(w_1)(s)- F_p(w_2)(s) \Big] ds \Bigg) \phi_p(x). \label{es3}
	\end{align}
	Combining \eqref{es1}, \eqref{es2}, \eqref{es3} and using the inequality $(a + b)^2 \le (1 + \theta)a^2 + \left(1 + \frac{1}{\theta} \right)b^2$
	for
	any real numbers $a, b$ and $\theta>0$, we get the following estimate for all $t \in (0,T)$
	\begin{align}
	\|	\mathcal{J}(w_1)(.,t)- \mathcal{J}(w_2)(.,t)\|^2
	&\le e^{-2tM_n^{2\beta}}   K^2 (1+\theta)t^2 \|w_1-w_2\|_1^2 \nn \\
	& + e^{-2tM_n^{2\beta}}   K^2\left(1+\frac{1}{\theta}\right)(T-t)^2 \|w_1-w_2\|_1^2.
	\end{align}
	By choosing $\theta=\frac{T-t}{t}$, we obtain
	\begin{align}
	e^{2tM_n^{2\beta}}	\|\mathcal J(w_1)(t)- \mathcal J(w_2)(t) \|^2 \le  K^2 T^2 \|w_1-w_2\|^2_1,~~\text{for all} ~~t \in (0,T).  \label{j1}
	\end{align}
	On other hand,  letting $t=T$ in \eqref{es2}, we deduce
	\begin{align}
	e^{2TM_n^{2\beta}} \|\mathcal J(w_1)(T)-\mathcal  J(w_2)(T) \|^2  \le K^2 T^2 \|w_1-w_2\|^2_1. \label{j2}
	\end{align}
	By letting $t=0$ in \eqref{es1}, we have
	\begin{align}
	\|\mathcal J(w_1)(0)-\mathcal  J(w_2)(0) \|^2 \le K^2 T^2 \|w_1-w_2\|^2_1. \label{j3}
	\end{align}
	Combining \eqref{j1}, \eqref{j2} and \eqref{j3}, we obtain
	\begin{align}
	e^{tM_n^{2\beta}} 	\|\mathcal J(w_1)(t)-\mathcal J(w_2)(t) \| \le KT \|w_1-w_2\|_1,~~0\le t \le T, \nn
	\end{align}
	which leads to \eqref{J}. Since $KT<1$, we can conclude that $\mathcal J$ is a contraction; by the Banach fixed point theorem, it follows that the equation $\mathcal J(w) = w$ has a
	unique solution $\widehat U_{M_n,n} \in C([0, T];L^2(\Omega))$.
\end{proof}

{\bf Part B.} Estimate $\mathbb{\bf E} \| \widehat U_{M_n,n}(\cdot,t)-u(\cdot,t)\|^2_{L^2(\Omega)}$.
We have
\begin{align}
 \widehat U_{M_n,n}(x,t)-u(x,t)& =\underbrace{\Phi_{M_n,n}( \widetilde u_T)(x,t) - \Phi_{M_n,n}( u_T)(x,t)}_{B_{1,M_n,n}(x,t) }+\underbrace{\sum_{p=0}^{M_n}  e^{(T-t)p^{2\beta}}  \widetilde {G}_{np} \phi_p(x)}_{B_{3,M,n}(x,t)}\nonumber\\
& -\underbrace{\sum_{p=0}^{M_n}   \Bigg[   \int_t^T e^{(s-t)p^{2\beta}}  \Big( F_p(\widehat U_{M_n,n})-F_p(u)(s)\Big)ds \Bigg]\phi_p(x)}_{{B_{7,M,n}(x,t)}}\nonumber\\
& +\underbrace{\sum_{p=M_n+1}^{\infty}   \Bigg[   \int_0^t e^{(s-t)p^{2\beta}}  \Big( F_p(\widehat U_{M_n,n})-F_p(u)(s)\Big)ds \Bigg]\phi_p(x)}_{{B_{8,M_n,n}(x,t)}}\nn\\
&-\underbrace{\sum_{p=M_n+1}^\infty  e^{-tp^{2\beta}}  u_p(0)\phi_p(x)}_{{B_{9,M_n,n}(x,t)} }.  \label{80}
\end{align}
This implies that
\begin{align}
\mathbb{\bf E} \| \widehat U_{M_n,n}(\cdot,t)-u(\cdot,t)\|^2_{L^2(\Omega)} &\le 5 \mathbb{\bf E} \| B_{1,M_n,n}(\cdot,t)\|^2_{L^2(\Omega)}+5\| B_{3,M_n,n}(\cdot,t)\|^2_{L^2(\Omega)}\nn\\
&+5 \mathbb{\bf E} \| B_{7,M_n,n}(\cdot,t)\|^2_{L^2(\Omega)}+5\mathbb{\bf E}\| B_{8,M_n,n}(\cdot,t)\|^2_{L^2(\Omega)}\nn\\
&+5\| B_{9,M_n,n}(\cdot,t)\|^2_{L^2(\Omega)}. \label{81}
\end{align}
Using H\"{o}lder's inequality and Lipschitz condition of $F$, we have
\begin{align}
\|B_{7,M_n,n}(\cdot, t)\|^2_{L^2(\Omega)}&= \sum_{p=0}^{M_n}   \Bigg[   \int_t^T e^{(s-t)p^{2\beta}}  \Big( F_p(\widehat U_{M_n,n})-F_p(u)(s)\Big)ds \Bigg]^2 \nonumber\\
& \le \sum_{p=0}^{M_n} \int_t^T e^{2(s-t)p^{2\beta}} \Big( F_p(\widehat U_{M_n,n})-F_p(u)(s)\Big)^2 ds \nonumber\\
&\le K \int_t^T e^{2(s-t) M_n^{2\beta}} \|\widehat U_{M_n,n}(.,s)-u(.,s)\|^2_{L^2(\Omega)}ds.  \label{82}
\end{align}
The term $\mathbb{\bf E}\|B_{7,M_n,n}(\cdot, t)\|^2_{L^2(\Omega)}$ is bounded by
\begin{equation}
\mathbb{\bf E}\|B_{7,M_n,n}(\cdot, t)\|^2_{L^2(\Omega)} \le  K \int_t^T e^{2(s-t)M_n^{2\beta}} \mathbb{\bf E} \|\widehat U_{M_n,n}(.,s)-u(.,s)\|^2_{L^2(\Omega)} ds. \label{83}
\end{equation}
By a similar way as above, we estimate  $\mathbb{\bf E}\|B_{8,M_n,n}(\cdot, t)\|^2_{L^2(\Omega)}$ as follows
\begin{equation}
\mathbb{\bf E}\|B_{8,M_n,n}(\cdot, t)\|^2_{L^2(\Omega)} \le  K \int_0^t e^{2(s-t)M_n^{2\beta}} \mathbb{\bf E} \|\widehat U_{M_n,n}(.,s)-u(.,s)\|^2_{L^2(\Omega)} ds. \label{84}
\end{equation}
Finally, we bound $\|B_{9,M_n,n}(\cdot, t)\|^2_{L^2(\Omega)}$ by
\begin{align}
\|B_{9,M_n,n}(\cdot, t)\|^2_{L^2(\Omega)}&= \sum_{p=M_n+1}^\infty e^{-tp^{2\beta}}  p^{-2 \gamma} \Big( p^{\gamma} u_p(0)\Big)^2\nn\\ &\le e^{-2t M_n^{2\beta}} M_n^{-2\beta \gamma} \sum_{p=M_n+1}^\infty \Big(p^{\gamma}  u_p(0)\Big)^2\nn\\
& \le M_n^{-2\beta \gamma} e^{-2t M_n^{2\beta}}  \|u(0)\|^2_{H^{\gamma}(\Omega)}. \label{85}
\end{align}
We estimate the  bounds for the other terms as in the proof of Theorem \ref{main-theorem-1}, hence we can  conclude that
\begin{align*}
\mathbb{\bf E} \| \widehat U_{M_n,n}(\cdot,t)-u(\cdot,t)\|^2_{L^2(\Omega)} &\le \Big(5\pi^2 V_{\max}^2 +C_1^2(\beta, \widetilde E_1)\Big) \frac{(M_n+1) e^{2(T-t) M_n^{2\beta}}}{n}\nn\\
&+5M_n^{-2\beta \gamma} e^{-2t M_n^{2\beta}}  \|u(0)\|^2_{H^{\gamma}(\Omega)}\nn\\
&+ 5K \int_0^T e^{2(s-t)M_n^{2\beta}} \mathbb{\bf E} \|\widehat U_{M_n,n}(.,s)-u(.,s)\|^2_{L^2(\Omega)} ds,
\end{align*}
where we combined \eqref{80}, \eqref{81}, \eqref{82},\eqref{83}, \eqref{84}, \eqref{85}.
Multiplying the latter inequality with $e^{2M_n^{2\beta} t}$ , we obtain
\begin{align*}
e^{2M_n^{2\beta} t}\mathbb{\bf E} \| \widehat U_{M_n,n}(\cdot, t)-u(\cdot,t)\|^2_{L^2(\Omega)} &\le \Big(5\pi^2 V_{\max}^2 +C_1^2(\beta, \widetilde E_1)\Big) \frac{(M_n+1) e^{2T M_n^{2\beta}}}{n}\nn\\
&+5M_n^{-2\beta \gamma}  \|u(0)\|^2_{H^{\gamma}(\Omega)}\nn\\
&+ 5K \int_0^T e^{2sM_n^{2\beta}} \mathbb{\bf E} \|\widehat U_{M_n,n}(.,s)-u(.,s)\|^2_{L^2(\Omega)} ds.
\end{align*}
Since $\widehat U_{M_n,n},~ u \in C([0, T];L^2(\Omega))$ we obtain that the function $e^{2M_n^{2\beta} t}\mathbb{\bf E} \| \widehat U_{M_n,n}(\cdot,t)-u(\cdot,t)\|^2_{L^2(\Omega)}$ is continuous on $ [0, T ]$. Therefore, there
exists a positive $$\widetilde A= \sup_{0 \le t \le T} e^{2M_n^{2\beta} t}\mathbb{\bf E} \|\widehat  U_{M_n,n}(\cdot,t)-u(\cdot,t)\|^2_{L^2(\Omega)}.$$ This implies that
\begin{equation*}
\widetilde A \le \Big(5\pi^2 V_{\max}^2 +C_1^2(\beta, \widetilde E_1)\Big) \frac{(M_n+1) e^{2T M_n^{2\beta}}}{n}
+5M_n^{-2\beta \gamma}  \|u(\cdot, 0)\|^2_{H^{\gamma}(\Omega)}+ 5kT \widetilde A.
\end{equation*}
Hence
\begin{align*}
&e^{2M_n^{2\beta} t}\mathbb{\bf E} \| \widehat U_{M_n,n}(\cdot,t)-u(\cdot,t)\|^2_{L^2(\Omega)} \nn\\
& \le \widetilde A
\le \frac{ \Big(5\pi^2 V_{\max}^2 +C_1^2(\beta, \widetilde E_1)\Big) \frac{(M_n+1) e^{2T M_n^{2\beta}}}{n}
	+5M_n^{-2\beta \gamma}  \|u(\cdot, 0)\|^2_{H^{\gamma}(\Omega)} }{1-5kT}.
\end{align*}

\section{Space fractional diffusion equation with randomly  perturbed time dependent coefficients}\label{section3}

\subsection{Problem setting and regularization method}

In this section, we consider the inverse problem for   space fractional diffusion equation with perturbed time dependent coefficients
\begin{equation}
\label{problem444}
\left\{\begin{array}{l l l}
u_t + a(t) (-\Delta )^\beta u & = F(u(x,t))+g(x,t), & \qquad (x,t) \in \Omega\times (0,T),\\
u_x(x,t)&=0, & \qquad x \in \partial {\Omega},\\
u(x,T) & = u_T(x), & \qquad x \in {\Omega},
\end{array}
\right.
\end{equation}
where $0<a(t) < a_0$ for some  positive number $a_0$. Here,
the source function
$F: \mathbb{R} \to
\mathbb{R}$ a locally Lipschitz function that satisfies:  for each $Q>0$ and
for any $u$, $v$ satisfying $|u|, |v| \leqslant Q$, there holds
\begin{equation}\label{dkff}
\left| {F(u) - F(v)} \right|
\leqslant K_F(Q)\left| {u - v} \right|,
\end{equation}
where
\[
K(Q): =
\sup \left\{ {\left| {\frac{{F(u) - F(v)}}{{u - v}}}
	\right|:\left| u \right|,\left| v \right|
	\leqslant Q,u \ne v}
\right\} <  + \infty.
\]
We note that the function $Q\to K(Q)$ is increasing and $\mathop {\lim }\limits_{Q \to
	+ \infty } K(Q) =  + \infty $. Next, for  the ease of  the reader, we describe
our regularized method and analysis.

First, we approximate $u(x,T)$ and  $g(x,t)$   by
\begin{equation}
\overline  w_{M_n,n}(x)= \dfrac{1}{n}\sum_{k=1}^n \widetilde u_T(x_k)+ \sum_{p=1}^{M_n} \Bigg[\dfrac{\pi}{n} \sum_{k=1}^n \widetilde u_T(x_k) \phi_p(x_k) \Bigg] \phi_p(x). \label{wMn}
\end{equation}
and
\begin{equation}
\overline g_{M_n,n}(x,t)= \dfrac{1}{n}\sum_{k=1}^n \widetilde g(x_k,t)+\sum_{p=1}^{M_n} \Big[   \dfrac{\pi}{n} \sum_{k=1}^n \widetilde g(x_k,t) \phi_p(x_k) \Big] \phi_p(x), \label{gMn}
\end{equation}
respectively.
Second,  we approximate $F$ by $\overline F_Q$ defined by
\[
\overline F_Q\left(u(x,t) \right)
=
\begin{cases}
F(Q), &\quad u(x,t) >Q,\\
F(u(x,t)), &\quad - Q \le u(x,t)\le Q, \\
F(-Q),&\quad u(x,t) < - Q.
\end{cases}
\]
for all $Q>0$.
Since $K$ is increasing function in $[0, +\infty)$, we choose a sequence $\{Q_n \}$ satisfying   $Q_n \to +\infty $ as $n \to +\infty$.  Using Lemma 1.1 of the paper \cite{Tuan}, we also have  the Lipschitz continuity of  the function $\overline F_Q$ from  the following lemma
	
\begin{lemma}[\cite{Tuan}]\label{F-Q-lipshitz}
			For $v_1. v_2 \in L^2(\Omega)$, we have
			\begin{equation}
			\|\overline F_{Q_n} (v_1)-\overline F_{Q_n}(v_2) \|_{L^2(\Omega)} \le 2K(Q_n) \|v_1-v_2\|_{L^2(\Omega)}.
			\end{equation}	
		\end{lemma}

Next, since $a(t)$ is noised by $ \overline a(t)$,  using the fractional Laplacian defined by the spectral theorem we have
$$a(t) (-\Delta )^\beta f =a(t) \sum_{p=0}^{\infty} p^{2\beta}  <f, \phi_p> \phi_p(x), ~f \in L^2(\Omega). $$
  We approximate the operator $ a(t) (-\Delta )^\beta$ by  regularized operator $ \overline a(t) (-\Delta )^\beta- \overline  a_0 \mathbb{\bf R}_{n,\beta} $. By  the   observations and steps above,
we present  the  following new  regularized problem using the quasi-reversibility method.

Assume that $\overline a(t) < a_0$  for all $t\in [0,T]$.
%{\color{red}Let $b(t)= a_0 - a(t)$ and $\overline b(t)=a_0 - \overline a(t)$. where do we use these definitions?}
We study the  following regularized problem { with  Neumann boundary condition}
\begin{equation}
\label{regularizedproblem}
\left\{%\begin{array}{l l l}
\begin{split}
&\frac{\partial W_{M_n,n} }{\partial t} + \overline  a(t) (-\Delta )^\beta W_{M_n,n}-a_0 {\bf \overline R}_{n,\beta}( W_{M_n,n} )\\
 &= \overline F_{Q_n} (W_{M_n,n}(x,t)) +\overline g_{M_n,n} (x,t), \ \   (x,t) \in \Omega\times (0,T),\\
 \frac{\partial W_{M_n,n}(x,t) }{\partial x} &=0,  \qquad x \in \partial {\Omega},\\
 W_{M_n,n}(x,T) & = \overline w_{M_n,n}(x),  \qquad x \in {\Omega},
%\end{array}
\end{split}
\right.
\end{equation}
where $	{\bf \overline R}_{n,\beta}$ is defined by
\begin{equation}\label{upper-tail-operator}
{\bf \overline R}_{n,\beta}(v)  = \sum_{p \ge M_n a_0^{-\frac{1}{2\beta}  }}^\infty     p^{2\beta} \big\langle v,\phi_p\big\rangle_{L^2(\Omega)} \phi_p(x),
\end{equation}
for any function $v \in L^2(\Omega)$.

%Now, we state the first main result in this section.

%\subsection{ The main results}

In the remainder of this  section we give  two  results  on  convergence rate  of $W_{M_n,n}$ to $u$. The first result in the next section concerns the error estimate in $L^2(\Omega)$. The second result in subsection \ref{section-H-beta-estimate} concern the error estimate in the higher Sobolev space $H^\beta(\Omega)$.
\subsection{ Error estimate in $L^2(\Omega)$}
\begin{theorem}\label{thm-time-dependent-1}
	Suppose that $M_n$
	satisfies
	\begin{equation}
	\lim_{n\to +\infty}  \frac{(M_n+1)e^{2T M_n^{2\beta}}}{n}\quad \text{bounded},
	\end{equation}
	and $\ep$ satisfies that
	\begin{equation}
	\lim_{\ep \to 0} e^{TM_n^{2\beta} } \ep \quad \text{bounded}. \label{noiselevel}
	\end{equation}	
	Choose $Q_n$  such that
	\begin{equation}
	\lim_{n \to +\infty} e^{2T K(Q_n) } e^{-2tM_n^{2\beta} } =0,~~0<t \le T.
	\end{equation}	
	Then the problem  \eqref{regularizedproblem} has unique solution $W_{M_n,n} \in C([0,T];L^2(\Omega))$.
	
	Furthermore, assume that $g, { u} \in L^\infty\lf([0,T];\tilde V(\Omega)\rt)$, where $\tilde V(\Omega)$ is defined  in \eqref{Vomega}.  Then we have
	\begin{equation}\label{eqn:time-dep-approx}
	\mathbb{\bf E}  \|W_{M_n,n}(\cdot,t)-{u}(\cdot,t)\|^2_{L^2(\Omega)} \le e^{(2K(Q_n) +4)T} e^{-2t M_n^\beta } { \Phi} (n,{ u},g,\delta,\ep) ,
	\end{equation}
	where
	\begin{align}
	&{ \Phi} (n,{ u},g,\delta,\ep)=\left(\pi^2 V_{\max}^2+ \big|\overline C (\delta,u) \big|^2+ \pi^2 T^3 \vartheta^2+T \big|\overline D (\delta,g)\big|^2 \right) \frac{(M_n+1)e^{2T M_n^{2\beta} }}{n}\nn\\
	&+ \ep^2 e^{2T M_n^{2\beta} } T^2 \|{ u}\|^2_{L^\infty [0,T];H^{2\beta}(\Omega))} +  \Big(\|{u}_T\|_{\tilde V(\Omega)}^2+ T \|g\|^2_{L^\infty([0,T]; \widetilde V(\Omega))}+T a_0^2 \|{u}\|_{L^\infty\lf([0,T];\tilde V(\Omega)\rt)}^2 \Big).\nn
	\end{align}
\end{theorem}

\begin{remark}  \label{mainremark}
	Let us choose $M_n$ as follows
	\begin{align}
	e^{2T M_n^{2\beta} }= n^{\sigma}, \quad 0<\sigma<1.
	\end{align}
	Then we have
	\begin{equation}
	M_n:= \left( \frac{\sigma}{2T} \log(n)  \right)^{\frac{1}{2\beta}}.
	\end{equation}
	
	\noindent Choose $\ep< \overline \theta < n \ep $, and  $0<\sigma< \frac{\log(\frac{\overline \theta}{\ep})}{\log n}  $. Then  $n^\sigma \ep^2 < \overline \theta$.
	Moreover, it is easy to check that $M_n$ satisfies the condition of Theorem \ref{thm-time-dependent-1}.
	We can choose $Q_n$ such that $ e^{2K(Q_n)T}n^{-\frac{\sigma t}{T}} $ bounded as $n \to \infty$.  In particular we can  choose $Q_n$ such that
	\begin{align}
	K(Q_n) \le \frac{1}{2T}  \log \Big( \log(n) \Big). \label{Qnn}
	\end{align}
We note that	the term $\log \Big( \log(n) \Big) \to +\infty$ as $n\to\infty$. The choice of $Q_n$ as in \ref{Qnn} is suitable since we recall that the function $Q\to K(Q)$ is increasing and $\mathop {\lim }\limits_{Q \to
		+ \infty } K(Q) =  + \infty $.\\
	Under the assumptions above we can deduce that the error $	\mathbb{\bf E}  \|W_{M_n,n}(\cdot,t)-{u}(\cdot,t)\|^2_{L^2(\Omega)} $ is of order $\log(n) n^{-\frac{\sigma t}{T}}$ for $t \in (0,T]$.
\end{remark}

{ %\color{red}
\begin{remark}
	The error in inequality \ref{eqn:time-dep-approx} is not useful for $t=0$. To get an approximation of $u(x,0)$, we give another    Lemma below.
\end{remark}
\begin{lemma}  \label{mainlemma}
Assume that $M_n$ satisfies
\begin{equation}
M_n^\beta > \frac{1}{T} \log(\frac{1}{T}) . \label{asss}
\end{equation}
Then there exists unique $t_n \in (0,T)$ such that
\begin{equation}
e^{-t_n M_n^\beta}= t_n.
\end{equation}	
	Assume that $u$ satisfies that $  \frac{\partial u(x,t)}{\partial t} \in L^\infty(0,T;L^2(\Omega))$. 	Choose $Q_n$  such that
	\begin{equation}
	\lim_{n \to +\infty} \frac{  e^{2T K(Q_n) }} {M_n^{\beta}}   =0. \label{Qn}
	\end{equation}
	 Then we have the following estimate
	 \begin{align}
	 \mathbb{\bf E}  \|W_{M_n,n}(\cdot,t_n)-{u}(\cdot,0)\|^2_{L^2(\Omega)} \le  2 { \Phi} (n,{ u},g,\delta,\ep)  e^{(2K(Q_n) +4)T} \frac{1}{M_n^\beta} +2\frac{1}{M_n^\beta}\Big\|   \frac{\partial u(\cdot,t)}{\partial t}  \Big\|^2_{L^\infty([0,T];L^2(\Omega))}.
	 \end{align}
\end{lemma}

\begin{remark}
	Suppose that  $Q_n$ satisfies \eqref{Qn}, then  the error $\mathbb{\bf E}  \|W_{M_n,n}(\cdot,t_n)-{u}(\cdot,0)\|^2_{L^2(\Omega)}$ is of order $\frac{  e^{2T K(Q_n) }} {M_n^{\beta}}$.
\end{remark}
	
}
We give a proof of Lemma \ref{mainlemma} by using the estimates in Theorem \ref{thm-time-dependent-1}.

{
\begin{proof}[ \bf Proof of Lemma \ref{mainlemma}]
	First, consider the function $\varphi(z)=	e^{-z M_n^\beta}-z $ for $z \in (0, T)$.	
	Note that $\varphi$ is decreasing function and
	$\varphi(0)= 1$. And
	\begin{equation}
	\varphi(T)= e^{-T M_n^\beta}-T <0
	\end{equation}
	by the assumption \eqref{asss}.
	This implies that the equation $\varphi(z)=0$ has unique solution $z_0 \in (0,T)$. Since $z_0$ depends on $n$, we can denote it by $t_n$. 	
	Using the inequality $e^{m} \ge m$ for $m>0$, we deduce that
	\begin{equation*}
	\frac{1}{t_n}= e^{t_n M_n^\beta} \ge t_n  M_n^\beta.
	\end{equation*}
	Hence
	\begin{equation}
	t_n \le \sqrt{\frac{1}{M_n^\beta}}.
	\end{equation}
	Now, we will consider the error $ \mathbb{\bf E}  \|W_{M_n,n}(\cdot,t_n)-{u}(\cdot,0)\|^2_{L^2(\Omega)}.$	
	First, using the triangle inequality and the inequality $(a_0+a_1)^2 \le 2 a_0^2+ 2a_1^2$ for any positive $a_0, a_1$,  we have
	\begin{align}
	\|W_{M_n,n}(x,t_n)-{u}(x,0)\|^2_{L^2(\Omega)} &\le \Bigg(  \|W_{M_n,n}(\cdot,t_n)-{u}(\cdot,t_n)\|_{L^2(\Omega)}+\|u(\cdot,t_n)-{u}(\cdot,0)\|_{L^2(\Omega)} \Bigg)^2\nn\\
	&\le  2 \|W_{M_n,n}(\cdot,t_n)-{u}(\cdot,t_n)\|^2_{L^2(\Omega)}+2\|u(\cdot,t_n)-{u}(\cdot,0)\|^2_{L^2(\Omega)}. \label{asss11}
	\end{align}	
	Since	\eqref{eqn:time-dep-approx} holds for any $t>0$, we obtain
	\begin{align}
	\mathbb{\bf E}  \|W_{M_n,n}(\cdot,t_n)-{u}(\cdot,t_n)\|^2_{L^2(\Omega)} \le  e^{(2K(Q_n) +4)T} e^{-2t_n M_n^\beta } { \Phi} (n,{ u},g,\delta,\ep). \label{asss12}
	\end{align}
	From the Newton-Leibniz formula, we get
	\begin{align}
	\|u(\cdot,t_n)-{u}(\cdot,0)\|^2_{L^2(\Omega)} &= \Big\| \int_0^{t_n} \frac{\partial u(\cdot, s)}{\partial s} ds \Big\|^2_{L^2(\Omega)} \le \Bigg(\int_0^{t_n} \Big\|  \frac{\partial u(\cdot,s)}{\partial s}  \Big\|_{L^2(\Omega)}ds\Bigg)^2\nn\\
	&\le  t_n^2 \Big\|   \frac{\partial u(\cdot,t)}{\partial t}  \Big\|^2_{L^\infty([0,T];L^2(\Omega))}. \label{asss13}
	\end{align}	
	Combining \eqref{asss11},\eqref{asss12}, \eqref{asss13}, we get
	\begin{align}
	\mathbb{\bf E}  \|W_{M_n,n}(\cdot,t_n)-{u}(\cdot,0)\|^2_{L^2(\Omega)} &\le 2 e^{(2K(Q_n) +4)T} e^{-2t_n M_n^\beta } { \Phi} (n,{ u},g,\delta,\ep)+  2t_n^2 \Big\|   \frac{\partial u(\cdot,t)}{\partial t}  \Big\|^2_{L^\infty([0,T];L^2(\Omega))} \nn\\
	&\le 2 { \Phi} (n,{ u},g,\delta,\ep)  e^{(2K(Q_n) +4)T} \frac{1}{M_n^\beta} +2\frac{1}{M_n^\beta}\Big\|   \frac{\partial u(\cdot,t)}{\partial t}  \Big\|^2_{L^\infty([0,T];L^2(\Omega))}.
	\end{align}
	since we have used to $e^{-t_n M_n^\beta}= t_n \le \sqrt{\frac{1}{M_n^\beta}}$.
\end{proof}
}

	To prove Theorem \ref{thm-time-dependent-1}, first we need some  Lemmas.
	\begin{lemma}
		The problem \eqref{regularizedproblem} has unique solution $ W_{M_n,n} \in C([0,T];L^2(\Omega)).$
	\end{lemma}
	\begin{proof}Let $b(t)= a_0 - a(t)$ and $\overline b(t)=a_0 - \overline a(t)$.
		The equation \eqref{regularizedproblem}  can be transformed to the following equation
		\begin{equation}
		\frac{\partial W_{M_n,n} }{\partial t} - \overline  b(t) (-\Delta )^\beta W_{M_n,n}+a_0 {\bf \overline R}_{n,\beta} W_{M_n,n}  = \overline F_{Q_n} (W_{M_n,n}(x,t)) +\overline g_{M_n,n} (x,t). \label{eq2}
		\end{equation}
		Set $\widehat W_{M_n,n} (x,t)= W_{M_n,n}(x,T-t)$ then we get
		\begin{equation}
		\frac{\partial \widehat  W_{M_n,n} }{\partial t} + \overline  b(t) (-\Delta )^\beta \widehat W_{M_n,n}  = a_0 {\bf \overline R}_{n,\beta} \widehat W_{M_n,n}-\overline F_{Q_n} (\widehat W_{M_n,n}(x,t)) -\overline g_{M_n,n} (x,t), \label{1problem}
		\end{equation}
		and the initial condition
		\begin{equation}
		\widehat	W_{M_n,n}(x,0)  = \overline w_{M_n,n}(x). \label{initial}
		\end{equation}
		For any function $v\in L^2(\Omega)$, set
$$\overline H( v) =a_0 {\bf \overline R}_{n,\beta} v-\overline F_{Q_n} (v(x,t)) -\overline g_{M_n,n} (x,t) .$$
		It is easy to check that $\overline H$  is  globally Lipschitz function. In fact, for $v_1, v_2 \in L^2(\Omega)$,
		we obtain using Lemma \ref{F-Q-lipshitz}
		\begin{align}
		\|\overline H( v_1)-\overline H( v_2) \|_{L^2(\Omega)} &\le  a_0	\|\mathbb{\bf \overline R}_{n,\beta} v_1-\mathbb{\bf \overline R}_{n,\beta} v_2 \|_{L^2(\Omega)}+	\|\overline F_{Q_n} (v_1)-\overline F_{Q_n}(v_2) \|_{L^2(\Omega)} \nn\\
		&\le M_n^{2\beta} \| v_1-V_2\|_{L^2(\Omega)}+ 2K(Q_n) \|v_1-v_2\|_{L^2(\Omega)}.
		\end{align}

		By taking the inner product of \eqref{eq2} with $\phi_p(x)$, we get
		\begin{align}
		\frac{d}{dt} \Big<\widehat	W_{M_n,n}(.,t), \phi_p  \Big> +\overline b
		(t)p^{2\beta} \Big<\widehat	W_{M_n,n}(.,t), \phi_p  \Big>= \Big<\overline H (\widehat	W_{M_n,n}(.,t)) ,\phi_p  \Big>.
		\end{align}
		By multiplying both sides of the latter equality with $\exp\Big( p^{2\beta}\int_0^t \overline b (s)ds \Big)$, and then taking the integral from $0$ and $t$,  we transform the above differential equation into the following nonlinear integral equation
		\begin{align}
		\widehat	W_{M_n,n}(x,t)&= \sum_{p=1}^\infty \exp\Big( -p^{2\beta}\int_0^t \overline b (s)ds \Big) \Big[ \Big<\overline w_{M_n,n}, \phi_p  \Big>  \Big]\phi_p(x)\nn\\
		&+\int_0^t \exp\Big( -p^{2\beta}\int_s^t \overline b (\xi)d\xi \Big) \Big[ \Big<\overline H (\widehat	W_{M_n,n}(.,s)) ,\phi_p  \Big> ds \Big]\phi_p(x). \label{integralequation}
		\end{align}
		Here $\widehat	W_{M_n,n}$  is the mild solution of  \eqref{1problem}-\eqref{initial}.
		The existence of mild solution of  nonlinear integral equation \eqref{integralequation} is proved similarly as in the proof of part 1 of Theorem \ref{main-theorem-1}.
	 \end{proof}

	\begin{lemma}\label{lemma-r-p-bounds}
		Recall the  definition of  $\mathbb{\bf \overline R}_{n,\beta}$ in equation \eqref{upper-tail-operator} and define $ \mathbb{\bf \overline P}_{n,\beta}$
		\begin{align}
		{\bf \overline P}_{n,\beta}(v)  &= a_0 \sum_{p < M_n a_0^{-\frac{1}{2\beta}  }}   p^{2\beta} \big\langle v,\phi_p\big\rangle_{L^2(\Omega)} \phi_p(x),
		\end{align}
		for any function $v \in L^2(\Omega)$. Define the following space of functions
		\begin{equation}
		{	 \widetilde V(\Omega) }:=\Big\{ \theta \in L^2(\Omega):\ \   \sum_{p=1}^\infty   p^{4\beta} e^{2Ta_0^{\beta} p^{2\beta}}  \big\langle \theta,\phi_p\big\rangle_{L^2(\Omega)}^2 <+\infty \Big\}. \label{Vomega}
		\end{equation}
		Then we have
		\begin{equation}
			\|\mathbb{\bf \overline P}_{n,\beta} v\|_{L^2(\Omega)} \le  M_n^{2\beta} \| v\|_{L^2(\Omega)}, \quad \text{for any}\quad v \in L^2(\Omega),
		\end{equation}
		and
		\begin{equation}
		a_0	\|\mathbb{\bf \overline R}_{n,\beta} v\|_{L^2(\Omega)} \le a_0  e^{-TM_n^{2\beta} } \|v\|_{ \widetilde V(\Omega) } \quad \text{for any}\quad v \in  \widetilde V(\Omega) .
		\end{equation}
		\end{lemma}

	\begin{proof}	
		We obtain
		\begin{align}
			\|\mathbb{\bf \overline P}_{n} v\|^2_{L^2(\Omega)}& \le a_0^2  \sum_{p < M_n a_0^{-\frac{1}{2\beta}  }}^\infty    p^{4\beta}  \big\langle v,\phi_p\big\rangle_{L^2(\Omega)}^2 \nn\\
		&\le M_n^{4\beta} \sum_{p < M_n a_0^{-\frac{1}{2\beta}  }}^\infty    \big\langle v,\phi_p\big\rangle_{L^2(\Omega)}^2  \nn\\
		&\le  M_n^{4\beta} \| v\|^2_{L^2(\Omega)}.
		\end{align}
		and
		\begin{align}
		a_0^2 \|\mathbb{\bf \overline R}_{n} v\|^2_{L^2(\Omega)}& \le a_0^2 \sum_{p \ge M_n a_0^{-\frac{1}{2\beta}  }}^\infty    \exp\Big(-2Ta_0 p^{2\beta}\Big) p^{4\beta}\exp\Big(2Ta_0 p^{2\beta}\Big) \big\langle v,\phi_p\big\rangle_{L^2(\Omega)}^2 \nn\\&\le a_0^2 e^{-2TM_n^{2\beta} } \sum_{p=1}^\infty   p^{4\beta} \exp\Big(2Ta_0 p^{2\beta}\Big)  \big\langle v,\phi_p\big\rangle_{L^2(\Omega)}^2 \nn\\
		&\le a_0^2  e^{-2TM_n^{2\beta} } \|v\|^2_{ \widetilde V(\Omega) }.
		\end{align}
	\end{proof}
	\begin{lemma}  \label{lemma3.6}
		Assume that ${ u}_T= u(.,T) \in \tilde V(\Omega)$ and $g \in L^\infty([0,T]; \widetilde V(\Omega)) $. Then the following estimates hold
		\begin{equation}
		\mathbb{\bf E} \|\overline w_{M_n,n}- { u}_T\|^2_{L^2(\Omega)} \le  \left(\pi^2 V_{\max}^2+ \big|\overline C (\delta,u) \big|^2 \right) \frac{M_n+1}{n}+ e^{-2TM_n^{2\beta}}\|{u}_T\|_{\tilde V(\Omega)}^2.
		\end{equation}
		and
		\begin{align}
		\mathbb{\bf E} \|\overline g_{M_n,n}(.,t)- g(.,t)\|^2_{L^2(\Omega)} \le \left( \pi^2 T^2 \vartheta^2+ \big|\overline D (\delta,g)\big|^2 \right)  \frac{M_n+1}{n}+e^{-2TM_n^{2\beta}} \|g\|_{L^\infty([0,T]; \widetilde V(\Omega))}, \label{129}
		\end{align}		
		where
		$$
		\overline C (\delta,u)= \max\Big(\frac{2}{4^\delta}, \frac{\sqrt{2}}{\sqrt{\pi}2^\delta } \Big)  \Big( \sum\limits_{l = 1}^\infty  \frac{1}{l^{\delta}} \Big)\|{ u}_T\|_{H^\delta(\Omega)},
		$$
		and
		$$
		\overline D (\delta,g)= \max\Big(\frac{2}{4^\delta}, \frac{\sqrt{2}}{\sqrt{\pi}2^\delta } \Big) \Big( \sum\limits_{l = 1}^\infty  \frac{1}{l^{\delta}} \Big) \|g\|_{L^\infty([0,T];H^\delta(\Omega))} .
		$$
		for any $\delta>1$.
	\end{lemma}
	\begin{proof}
		In this proof, we will find the estimate between $u_T= u(.,T)$ and the approximate function $w_{M_n,n}$ defined in
		\eqref{wMn}.
		By  the formula of $u_T$  in Lemma \ref{lemma2.3}, we get
		\begin{align}
		u_T(x)&= \sum_{p=0}^\infty \Big< u_T, \phi_p \Big>\phi_p(x)\nn\\
		&=\dfrac{1}{n}\sum_{k=1}^n u_T(x_k) - \widetilde {G}_{n0}+ \sum_{p=1}^{M_n} \Bigg( \dfrac{\pi}{n} \sum_{k=1}^n u_T(x_k) \phi_p(x_k) -  \widetilde {G}_{np} \Bigg)\phi_p(x)\nn\\
		&+ \sum_{p=M_n+1}^\infty \Big< u_T, \phi_p \Big>\phi_p(x).
		\end{align}
		This together with \eqref{wMn} and the fact that $	\widetilde u_T(x_k)=u_T(x_k)+\sigma_k\epsilon_k$ gives
		\begin{align}
		\|\overline  w_{M_n,n}- { u}_T\|^2_{L^2(\Omega)} \le \Big[ \dfrac{1}{n}\sum_{k=1}^n \sigma_k \epsilon_k - \widetilde {G}_{n0} \Big]^2&+\sum_{p=1}^{M_n} \Bigg[  \dfrac{\pi}{n} \sum_{k=1}^n \sigma_k \epsilon_k \phi_p(x_k) -  \widetilde {G}_{np} \Bigg]^2\nonumber\\\
		&+ \sum_{p=M_n+1}^\infty \Big<{ u}_T, \phi_p\Big>^2. \label{error}
		\end{align}	
		This implies that
		\begin{align}
		\mathbb{\bf E} \|\overline w_{M_n,n}- { u}_T\|^2_{L^2(\Omega)}&\le \frac{1}{n^2} \sum_{k=1}^n \sigma_k^2 \mathbb{\bf E}\epsilon_k^2 + \Big|\widetilde {G}_{n0}\Big|^2+ \Bigg[\frac{\pi^2}{n^2}\sum_{p=1}^{M_n} \sum_{k=1}^n \sigma_k^2  \mathbb{\bf E}\epsilon_k^2 \phi_p^2(x_k)+ \sum_{p=1}^{M_n} \Big|\widetilde {G}_{np}\Big|^2\Bigg]\nn\\
		&+\sum_{p=M_n+1}^\infty \Big<{ u}_T, \phi_p\Big>^2. \label{estimate0}
		\end{align}		
		Further, by the formula
		\begin{equation}
		\|{ u}_T\|_{H^\delta(\Omega)}^2= \sum_{p=0}^\infty p^{2\delta} \Big<  { u}_T, \phi_p\Big>^2
		\end{equation}
		we obtain that
		for all $p \in \mathbb{N}, p \ge 1$
		\begin{align}
		\left|\Big<  { u}_T, \phi_p \Big> \right|  \le \frac{\|{ u}_T\|_{H^\delta(\Omega)}}{p^{\delta }}. \label{estimate11}
		\end{align}
		Now, we estimate the term $\widetilde G_{np} $ for $p \ge 0$. By inequality  \eqref{estimate11}, we have the bound for $\widetilde G_{n0}$
		\begin{equation}
		\widetilde G_{n0} \le \sqrt{\dfrac{2}{\pi}}\sum_{l=1}^\infty \left| {\langle { u}_T,{\phi _{2ln}}\rangle } \right|
		\le \sqrt {\frac{2}{\pi }} \sum\limits_{l = 1}^\infty  \frac{ \|{ u}_T\|_{H^\delta(\Omega)} }{2^{\delta}l^{\delta} n^{\delta}} ,\label{iq2}
		\end{equation}
		and the bound for $\widetilde G_{np}$
		\begin{align}
		\widetilde G_{np} &\le \sum_{l=1}^\infty  \Bigg|  \Big \langle { u}_T, \phi_{p+2ln} \Big\rangle + \Big \langle { u}_T, \phi_{-p+2ln} \Big\rangle\Bigg|\nonumber\\
		&\le  \|{ u}_T\|_{H^\delta(\Omega)} \Bigg[\sum_{l=1}^\infty \frac{1}{(p+2ln)^{\delta}} +\sum_{l=1}^\infty \frac{1}{(-p+2ln)^{\delta}}\Bigg]\nonumber\\
		& \le \frac{2\|{ u}(.,T)\|_{H^\delta(\Omega)}}{4^\delta} \Big( \sum\limits_{l = 1}^\infty  \frac{1}{l^{\delta}} \Big) \frac{1}{n^{\delta}}. \label{iq3}
		\end{align}	
		Since $\delta >1$, we know that the series   $\sum\limits_{l = 1}^\infty  \frac{1}{l^{\delta}}$ converges. Let us denote
		$$
		\overline C (\delta,u)= \max\Big(\frac{2}{4^\delta}, \frac{\sqrt{2}}{\sqrt{\pi}2^\delta } \Big)  \Big( \sum\limits_{l = 1}^\infty  \frac{1}{l^{\delta}} \Big)\|{ u}_T\|_{H^\delta(\Omega)}.
		$$
		By  \eqref{iq2} and \eqref{iq3}, we get
		\begin{equation}
		\widetilde G_{np} \le  \frac{\overline C (\delta,u)}{n^{\delta}},\quad \text{for all }\quad p \ge 0.  \label{Gnp}
		\end{equation}
		It follows from  \eqref{estimate0} that
		\begin{align}
		&\mathbb{\bf E} \|\overline w_{M_n,n}- {u}_T\|^2_{L^2(\Omega)} \nn\\
		&\le  \frac{ V_{\max}^2 (\pi^2 M_n+1)}{n} +  \frac{(M_n+1)\big|\overline C (\delta,u) \big|^2}{n^{2\delta}}+\sum_{p=M_n+1}^\infty \Big<{u}_T, \phi_p\Big>^2\nn\\
		&\le  \frac{ V_{\max}^2 (\pi^2 M_n+1)}{n} +  \frac{(M_n+1)\big|\overline C (\delta,u) \big|^2}{n^{2\delta}}+ e^{-2TM_n^2\beta}\sum_{p=M_n+1}^\infty e^{2T p^{2\beta}} \Big<{u}_T, \phi_p\Big>^2\nn\\
		&\le  \left(\pi^2 V_{\max}^2+ \big|\overline C (\delta,u) \big|^2 \right) \frac{M_n+1}{n}+ e^{-2TM_n^{2\beta}}\|{u}_T\|_{\tilde V(\Omega)}^2 .
		\end{align}
		Assume that $g \in L^\infty([0,T];H^\delta(\Omega)) $ and let us define
		$$
		\overline D (\delta,g)= \max\Big(\frac{2}{4^\delta}, \frac{\sqrt{2}}{\sqrt{\pi}2^\delta } \Big) \Big( \sum\limits_{l = 1}^\infty  \frac{1}{l^{\delta}} \Big) \|g\|_{L^\infty([0,T];H^\delta(\Omega))} .
		$$
		In  a similar way, we show that
		\begin{equation}
		\widetilde H_{np}(t) \le  \frac{\overline D (\delta,g)}{n^{\delta}},\quad \text{for all }\quad p \ge 0, t \in [0,T].  \label{Hnp}
		\end{equation}
		and
		\begin{align}
		\|\overline g_{M_n,n}(.,t)- g(.,t)\|^2& \le \Big[ \dfrac{1}{n}\sum_{k=1}^n \vartheta \xi_k(t)- \widetilde {H}_{n0}(t) \Big]^2+\sum_{p=1}^{M_n} \Bigg[  \dfrac{\pi}{n} \sum_{k=1}^n \vartheta \xi_k(t)\phi_p(x_k) -  \widetilde {H}_{np}(t) \Bigg]^2\nonumber\\\
		&+ \sum_{p=M_n+1}^\infty \Big<g(\cdot,t), \phi_p\Big>^2.
		\end{align}		
		From the properties of Brownian motion, we known that $\mathbb{\bf E}[\xi_{i}(t)\xi_{k}(t)] = 0$ for $k \ne i$ and $\mathbb{\bf E}\xi_{k}^2(t) = t$. By the H\"{o}lder inequality, we obtain
{
		\begin{align}
		\mathbb{\bf E} \|\overline g_{M_n,n}(.,t)- g(.,t)\|^2& \le  \dfrac{1}{n^2}\sum_{k=1}^n \vartheta^2 \mathbb{\bf E} \xi_k^2(t)+\sum_{p=1}^{M_n}   \dfrac{\pi^2}{n^2} \sum_{k=1}^n \vartheta^2 \mathbb{\bf E}\xi_{k}^2(t)\phi_p^2(x_k) + \sum_{p=0}^{M_n} \big| \widetilde {H}_{np}(t)\big|^2 \nn\\
		&+ e^{-2TM_n^{\beta}} \sum_{p=M_n+1}^\infty e^{2T p^{2\beta}} \Big<g(\cdot,t), \phi_p\Big>^2\nn\\
		&\le \frac{T \vartheta^2 (\pi^2 M_n+1)}{n}+\frac{(M_n+1)\big|\overline D (\delta,g)\big|^2}{n^{2\delta}} + e^{-2TM_n^{2\beta}} \|g\|^2_{L^\infty([0,T]; \widetilde V(\Omega))} \nn\\
		&\le \left( \pi^2 T \vartheta^2+ \big|\overline D (\delta,g)\big|^2 \right)  \frac{M_n+1}{n}+e^{-2TM_n^{2\beta}} \|g\|^2_{L^\infty([0,T]; \widetilde V(\Omega))}.
		\end{align}		
}
	\end{proof}
	
\begin{proof}[ \bf Proof of Theorem \ref{thm-time-dependent-1}]
\noindent	We now return the proof of Theorem.
	The main equation in \eqref{problem444} can be rewritten as follows
	\begin{equation}
	\frac{\partial { u} }{\partial t} +   \overline  a(t) (-\Delta )^\beta { u} = F({ u}(x,t))+ g(x,t)+ \Big(\overline  a(t)-  a(t)  \Big) (-\Delta )^\beta  { u}  .
	\end{equation}
	For $\rho_n>0$, we put $ \mathbf{ Y}_{M_n,n}(x,t)=e^{\rho_n(t-T)}\Big[ W_{M_n,n}(x,t)-{u}(x,t)\Big].$ Then, from the last two equalities,  a simple computation gives
	\begin{align}
	\frac{\partial \mathbf{ Y}_{M_n,n}}{\partial t} & - \overline  b(t) (-\Delta )^\beta \mathbf{ Y}_{M_n,n}-\rho_n \mathbf{ Y}_{M_n,n} \nn\\
	&= -e^{\rho_n(t-T)} a_0 {\bf \overline P}_{n,\beta}  \mathbf{ Y}_{M_n,n}+e^{\rho_n(t-T)} a_0  {\bf \overline R}_{n,\beta}  {u} - e^{\rho_n(t-T)}\Big(\overline  a(t)-  a(t)  \Big)(-\Delta )^\beta  { u}  \nn\\
	&\quad + e^{\rho_n (t-T)}\Big[\overline F_{Q_n}(W_{M_n,n}(x,t))- F\big(u(x,t)\big)\Big]\nn\\
	&\quad + e^{\rho_n (t-T)}\Big[\overline g_{M_n,n}(x,t))- g(x,t)\Big],\quad (x,t) \in \Omega\times (0,T), \label{difference}
	\end{align}
	and
	$$ \frac{\partial \mathbf{ Y}_{M_n,n} }{\partial x}|_{\partial \Omega}=0, ~ \mathbf{ Y}_{M_n,n}(x,T)=\overline w_{M_n,n}(x)- { u}_T(x) .$$
	Here, we note that
	\begin{align}
	\Big<v, (-\Delta )^\beta v\Big>_{L^2(\Omega)}&= \int_{\Omega}  \Big( \sum_{p=0}^\infty <v, \phi_p>\phi_p(x) \Big) \Big( \sum_{p=0}^\infty p^{2\beta} <v, \phi_p>\phi_p(x) \Big)dx \nn\\
	&= \sum_{p=0}^\infty p^{2\beta} <v, \phi_p>^2 =\|v\|^2_{H^\beta(\Omega)} .
	\end{align}
	By taking the inner product of the  two sides of the latter equality with $\mathbf{Y}_{M_n,n}$
	one deduces that
	\begin{align} \label{3J1}
	\frac{1}{2} \frac{d}{dt}  \|\mathbf{Y}_{M_n,n}(\cdot, t)\|^2_{L^2(\Omega)} &-\overline b(t) \|\mathbf{Y}_{M_n,n}(\cdot, t)\|^2_{H^\beta(\Omega)}- \rho_n \|\mathbf{Y}_{M_n,n}(\cdot, t)\|^2_{L^2(\Omega)}\nn\\
	& = \underbrace{ \Big<	-e^{\rho_n(t-T)} a_0 {\bf \overline P}_{n,\beta}  \mathbf{ Y}_{M_n,n}, \mathbf{Y}_{M_n,n} \big>_{L^2(\Omega)}}_{=:\widetilde{\mathcal{J}}_{1,n}} + \underbrace{ \Big< e^{\rho_n(t-T)} a_0  {\bf \overline R}_{n,\beta} {u}, \mathbf{Y}_{M_n,n} \Big>_{L^2(\Omega)}}_{=:\widetilde{\mathcal{J}}_{2,n}} \nn\\
	&+ \underbrace{\Big< - e^{\rho_n(t-T)}\Big(\overline  a(t)-  a(t)  \Big)(-\Delta )^\beta  { u} , \mathbf{Y}_{M_n,n} \Big>_{L^2(\Omega)}}_{=:\widetilde{\mathcal{J}}_{3,n}} \nn\\
	&+\underbrace{\Big< e^{\rho_n (t-T)}\lf[\overline F_{Q_n}(W_{M_n,n}(\cdot,t))- F\big(u(\cdot,t)\big)\rt], \mathbf{Y}_{M_n,n} \Big>_{L^2(\Omega)}}_{=:\widetilde{\mathcal{J}}_{4,n}}\nn\\
	&+\underbrace{\Big<e^{\rho_n (t-T)}\lf[\overline g_{M_n,n}(\cdot,t))- g(\cdot,t)\rt], \mathbf{Y}_{M_n,n} \Big>_{L^2(\Omega)}}_{=:\widetilde{\mathcal{J}}_{5,n}}	.
	\end{align}
	First, thanks to  Lemma \ref{lemma-r-p-bounds}, we bound   $\widetilde{\mathcal{J}}_{1,n}$ using the Cauchy-Schwartz inequality as follows
	\begin{align} \label{J1}
	\big|\widetilde{\mathcal{J}}_{1,n} \big| &\leq \| \mathbb{\bf P}_{n} \mathbf{Y}_{M_n,n}\|_{L^2(\Omega)}  \| \mathbf{Y}_{M_n,n} (\cdot,t) \|_{L^2(\Omega)}\le   M_n^{2\beta} \| \mathbf{Y}_{M_n,n} (\cdot,t) \|_{L^2(\Omega)}^2.
	\end{align}
	Using Lemma \ref{lemma-r-p-bounds} and Cauchy-Schwartz inequality,   the term  $\widetilde{\mathcal{J}}_{2,n}$ can be  estimated by
	\begin{align} \label{J2}
	\big|\widetilde{\mathcal{J}}_{2,n}\big|
	& \leqq \frac{1}{2} e^{2\rho_n(t-T)}  a_0^2  e^{-TM_n^{2\beta} } \|{u}\|_{L^\infty\lf([0,T];\tilde V(\Omega)\rt)}^2   +\frac{1}{2} \|\mathbf{Y}_{M_n,n}(\cdot,t)\|^2_{L^2(\Omega)} \nn\\
	&\leq \frac{1}{2}  a_0^2  e^{-2TM_n^{2\beta} } \|{u}\|_{L^\infty\lf([0,T];\tilde V(\Omega)\rt)}^2   +\frac{1}{2} \|\mathbf{Y}_{M_n,n}(\cdot,t)\|^2_{L^2(\Omega)},
	\end{align}
	From the inequality $2\langle a_1, a_2\rangle_{L^2(\Omega)} \leq \|a_1\|^2_{L^2(\Omega)} + \|a_2\|^2_{L^2(\Omega)}$ for any  $a_i \in L^2(\Omega) , ~(i=1,2)$, we infer
	\begin{align} \label{J3}
	\big|\widetilde{\mathcal{J}}_{3,n}\big| &=\lf|\Big< - e^{\rho_n(t-T)}\Big(\overline  a(t)-  a(t)  \Big)(-\Delta )^\beta  { u} , \mathbf{Y}_{M_n,n} \Big>_{L^2(\Omega)}\rt| \nn\\
	&\leq \frac{1}{2} e^{2\rho_n(t-T)}  \Big(\overline  a(t)-  a(t)  \Big)^2   \lf\| (-\Delta )^\beta  { u} \rt\|_{L^2(\Omega)}^2 + \frac{1}{2} \|\mathbf{Y}_{M_n,n}(\cdot,t)\|^2_{L^2(\Omega)}   \nn\\
	&\le 	\frac{1}{2}  e^{2\rho_n(t-T)}  \Big(\overline  a(t)-  a(t)  \Big)^2  \|{ u}\|^2_{L^\infty ([0,T];H^{2\beta}(\Omega))}+\frac{1}{2} \|\mathbf{Y}_{M_n,n}(\cdot,t)\|^2_{L^2(\Omega)} .
	\end{align}
	Finally, since $\lim_{n \rightarrow +\infty} Q_n=+\infty$, for a sufficiently large $n >0$ such that $Q_n \geq \|{u}\|_{L^\infty([0,T];L^2(\Omega))}$.  Moreover,  we have $\overline F_{Q_n}(u(x,t))=F(u(x,t))$. Using the global Lipschitz property of $
	\overline F_{Q_n}$, one similarly has for $\big|\widetilde{\mathcal{J}}_{4,n}\big|$ the fact that
	\begin{align} \label{J4}
	\big|\widetilde{\mathcal{J}}_{4,n}\big|&= \Big| \Big< e^{\rho_n (t-T)}\lf[\overline F_{Q_n}(W_{M_n,n}(\cdot,t))- F\big(u(\cdot,t)\big)\rt], \mathbf{Y}_{M_n,n} \Big>_{L^2(\Omega)}\Big|\nn\\
	& \leqq  \lf\|e^{\rho_n (t-T)}\lf[\overline F_{Q_n}(W_{M_n,n}(\cdot,t))- \overline F_{Q_n}\big(u(x,t)\big)\rt]\rt\|_{L^2(\Omega)}~\|\mathbf{Y}_{M_n,n}\|_{L^2(\Omega)}\nn\\
	&\leq  2K(Q_n) \|\mathbf{Y}_{M_n,n}(\cdot,t)\|_{L^2(\Omega)}^2.
	\end{align}
	The term $\big|\widetilde{\mathcal{J}}_{5,n}\big|$ can be bounded by
	\begin{align} \label{J41}
	\big|\widetilde{\mathcal{J}}_{5,n}\big|&= \Big| \Big<e^{\rho_n (t-T)}\lf[\overline g_{M_n,n}(\cdot,t))- g(\cdot,t)\rt], \mathbf{Y}_{M_n,n} \Big>_{L^2(\Omega)}\Big|\nn\\
	& \leq \frac{1}{2}e^{2\rho_n (t-T)} \lf\|\overline g_{M_n,n}(\cdot,t))- g(\cdot,t)\rt\|^2_{L^2(\Omega)}+ \frac{1}{2}\|\mathbf{Y}_{M_n,n}\|^2_{L^2(\Omega)}.
	\end{align}
	Combining   \eqref{3J1}, \eqref{J1}, \eqref{J2}, \eqref{J3} and  \eqref{J41}    gives
	\begin{align}
	\frac{1}{2} \frac{d}{dt}  \|\mathbf{Y}_{M_n,n}(\cdot, t)\|^2_{L^2(\Omega)} &-\rho_n \|\mathbf{Y}_{M_n,n}(\cdot, t)\|^2_{L^2(\Omega)}\nn\\
	&\geq - M_n^{2\beta}  \| \mathbf{Y}_{M_n,n} (\cdot,t) \|_{L^2(\Omega)}^2 -\frac{1}{2}  a_0^2  e^{-2TM_n^{2\beta} } \|{u}\|_{L^\infty\lf([0,T];\tilde V(\Omega)\rt)}^2 \nn\\
	&\quad -\frac{1}{2} \|\mathbf{Y}_{M_n,n}(\cdot,t)\|^2_{L^2(\Omega)} -  e^{2\rho_n(t-T)}  \Big(\overline  a(t)-  a(t)  \Big)^2  \|{ u}\|^2_{L^\infty ([0,T];H^{2\beta}(\Omega))}\nn\\
	& \quad - \|\mathbf{Y}_{M_n,n}(\cdot,t)\|^2_{L^2(\Omega)}- 2 K(Q_n) \|\mathbf{Y}_{M_n,n}(\cdot,t)\|_{L^2(\Omega)}^2\nn\\
	& \quad-\frac{1}{2}e^{2\rho_n (t-T)} \lf\|\overline g_{M_n,n}(\cdot,t))- g(\cdot,t)\rt\|^2_{L^2(\Omega)}- \frac{1}{2}\|\mathbf{Y}_{M_n,n}\|^2_{L^2(\Omega)}.
	\end{align}
	By taking the integral from $t$ to $T$ and by a simple calculation yields
	\begin{align}
	&\|\mathbf{Y}_{M_n,n}(\cdot, T)\|^2_{L^2(\Omega)}  -\|\mathbf{Y}_{M_n,n}(\cdot, t)\|^2_{L^2(\Omega)} \nn\\
	&\quad \quad+ \int_t^T \lf( a_0^2  e^{-2TM_n^{2\beta} }  \|{u}\|_{L^\infty\lf([0,T];\tilde V(\Omega)\rt)}^2  + \Big(\overline  a(s)-  a(s)  \Big)^2  \|{ u}\|^2_{L^\infty [0,T];H^{2\beta}(\Omega))}\rt) ds \nn\\
	&\quad \quad+\int_t^T  e^{2 \rho_n (s-T)} \lf\|\overline g_{M_n,n}(\cdot,s))- g(\cdot,s)\rt\|^2_{L^2(\Omega)} ds\nn\\
	&\quad \quad \geq \int_t^T \Big(2\rho_n- 2M_n^{2\beta} - 4K(Q_n) -4 \Big)\|\mathbf{Y}_{M_n,n}(\cdot, s) \|^2_{L^2(\Omega)} ds.
	\end{align}
	Let us choose $\rho_n=M_n^{2\beta}$.	This leads to
	\begin{align*}
	&e^{2 \rho_n (t-T)}	\|W_{M_n,n}(\cdot,t)-{u}(\cdot,t)\|^2_{L^2(\Omega)}\nn\\
	& \quad \quad  \leq \| \overline w_{M_n,n}(\cdot)- u_T(\cdot) \|^2_{L^2(\Omega)}+ T a_0^2e^{-2T M_n^\beta} \|{u}\|_{L^\infty\lf([0,T];\tilde V(\Omega)\rt)}^2 \nn\\
	& \quad \quad+ T \sup_{0 \le t \le T}\lf\|\overline g_{M_n,n}(\cdot,t))- g(\cdot,t)\rt\|^2_{L^2(\Omega)} + {\|{ u}\|^2_{L^\infty ([0,T];H^{2\beta}(\Omega))}\int_0^T \Big(\overline  a(s)-  a(s)  \Big)^2 ds } \nn\\
	&\quad \quad + \Big(4K(Q_n) +4\Big) \int_t^T e^{2 \rho_n (s-T)} \|W_{M_n,n}(\cdot,s)-{u}(\cdot,s)\|^2_{L^2(\Omega)}ds .
	\end{align*}
	Hence, we obtain
	\begin{align}
	&e^{2 \rho_n (t-T)}	\mathbb{\bf E}  \|W_{M_n,n}(\cdot,t)-{u}(\cdot,t)\|^2_{L^2(\Omega)}  \nn\\
	&\quad \quad  \leq \mathbb{\bf E}  \| \overline w_{M_n,n}(\cdot)- u_T(\cdot) \|^2_{L^2(\Omega)}+  T a_0^2e^{-2T M_n^\beta} \|{u}\|_{L^\infty\lf([0,T];\tilde V(\Omega)\rt)}^2   \nn\\
	&\quad \quad +\int_t^T e^{2 \rho_n (s-T)} \mathbb{\bf E} \lf\|\overline g_{M_n,n}(\cdot,s))- g(\cdot,s)\rt\|^2_{L^2(\Omega)}ds+ {\epsilon^2\|{ u}\|^2_{L^\infty ([0,T];H^{2\beta}(\Omega))}\int_0^T \mathbb{\bf E} \Big|\overline  {\xi}(s)\Big|^2 ds }  \nn\\
	&\quad \quad + \Big(4K(Q_n) +4\Big) \int_t^T e^{2 \rho_n (s-T)} \mathbb{\bf E}  \|W_{M_n,n}(\cdot,s)-{u}(\cdot,s)\|^2_{L^2(\Omega)}ds .
	\end{align}	
	In the above we have used the fact that
	$
	\overline a(t)- a(t)= \ep \overline  {\xi}(t)
	$
	and $\mathbb{\bf E} \Big|\overline  {\xi}(t)\Big|^2 =t$.
	Using  the second inequality of  Lemma \ref{lemma3.6} and noting that $e^{2\rho_n(s-T)} \le 1$ for $0 \le s \le T$, we have the following estimate
	\begin{align}
	&\int_t^T e^{2 \rho_n (s-T)} \mathbb{\bf E} \lf\|\overline g_{M_n,n}(\cdot,s))- g(\cdot,s)\rt\|^2_{L^2(\Omega)}ds \nn\\
	& \le \int_t^T e^{2 \rho_n (s-T)} \Bigg[\left( \pi^2 T^2 \vartheta^2+ \big|\overline D (\delta,g)\big|^2 \right)  \frac{M_n+1}{n}+e^{-2TM_n^{\beta}} \|g\|_{L^\infty([0,T]; \widetilde V(\Omega))}\Bigg]ds \nn\\
	&\le \int_t^T  \Bigg[\left( \pi^2 T^2 \vartheta^2+ \big|\overline D (\delta,g)\big|^2 \right)  \frac{M_n+1}{n}+e^{-2TM_n^{\beta}} \|g\|_{L^\infty([0,T]; \widetilde V(\Omega))}\Bigg]ds \nn\\
	&=\Bigg[\left( \pi^2 T^2 \vartheta^2+ \big|\overline D (\delta,g)\big|^2 \right)  \frac{M_n+1}{n}+e^{-2TM_n^{\beta}} \|g\|_{L^\infty([0,T]; \widetilde V(\Omega))}\Bigg] (T-t)\nn\\
	&\le \left( \pi^2 T^3 \vartheta^2+T \big|\overline D (\delta,g)\big|^2 \right)  \frac{M_n+1}{n}+Te^{-2TM_n^{\beta}} \|g\|_{L^\infty([0,T]; \widetilde V(\Omega))}.
	\end{align}
	From the observations above and by using  the first inequality in Lemma \ref{lemma3.6}, we conclude that
	\begin{align}
	&e^{2 \rho_n (t-T)}	\mathbb{\bf E}  \|W_{M_n,n}(\cdot,t)-{u}(\cdot,t)\|^2_{L^2(\Omega)}  \nn\\
	&\quad \quad \leq \left(\pi^2 V_{\max}^2+ \big|\overline C (\delta,u) \big|^2+ \pi^2 T^3 \vartheta^2+T \big|\overline D (\delta,g)\big|^2 \right) \frac{M_n+1}{n}+ \ep^2 T^2 \|{ u}\|^2_{L^\infty( [0,T];H^{2\beta}(\Omega))}\nn\\
	&\quad \quad + e^{-2TM_n^{\beta}} \Big(\|{u}_T\|_{\tilde V(\Omega)}^2+ T \|g\|^2_{L^\infty([0,T]; \widetilde V(\Omega))}+T a_0^2 \|{u}\|_{L^\infty\lf([0,T];\tilde V(\Omega)\rt)}^2 \Big)\nn\\
	&\quad \quad + \Big(4K(Q_n) +4\Big) \int_t^T e^{2 \rho_n (s-T)} \mathbb{\bf E}  \|W_{M_n,n}(\cdot,s)-{u}(\cdot,s)\|^2_{L^2(\Omega)}ds  .
	\end{align}	
	Multiplying both sides by $e^{2 T M_n^{2\beta} }$ , we obtain	
	\begin{align*}
	&e^{2 t M_n^{2\beta} }	\mathbb{\bf E}  \|W_{M_n,n}(\cdot,t)-{u}(\cdot,t)\|^2_{L^2(\Omega)}  \nn\\
	& \leq \left(\pi^2 V_{\max}^2+ \big|\overline C (\delta,u) \big|^2+ \pi^2 T^3 \vartheta^2+T \big|\overline D (\delta,g)\big|^2 \right) \frac{(M_n+1)e^{2T M_n^{2\beta} }}{n}+ \ep^2 e^{2T M_n^{2\beta} } T^2 \|{ u}\|^2_{L^\infty ([0,T];H^{2\beta}(\Omega))}\nn\\
	& +  \Big(\|\textbf{u}_T\|_{\tilde V(\Omega)}^2+ T \|g\|^2_{L^\infty([0,T]; \widetilde V(\Omega))}+T a_0^2 \|{u}\|_{L^\infty\lf([0,T];\tilde V(\Omega)\rt)}^2 \Big)\nn\\
	& + \Big(4K(Q_n) +4\Big) \int_t^T e^{2s M_n^\beta } \mathbb{\bf E}  \|W_{M_n,n}(\cdot,s)-{u}(\cdot,s)\|^2_{L^2(\Omega)}ds  .
	\end{align*}	
	Applying Gronwall's inequality, we get
	\begin{align*}
	&e^{2t M_n^{2\beta} }	\mathbb{\bf E}  \|W_{M_n,n}(\cdot,t)-{u}(\cdot,t)\|^2_{L^2(\Omega)} \le e^{(4K(Q_n) +4)T} { \Phi} (n,{ u},g,\delta,\ep) ,
	\end{align*}
	where
	\begin{align}
	&{\Phi} (n,{ u},g,\delta,\ep)=\left(\pi^2 V_{\max}^2+ \big|\overline C (\delta,u) \big|^2+ \pi^2 T^3 \vartheta^2+T \big|\overline D (\delta,g)\big|^2 \right) \frac{(M_n+1)e^{2T M_n^{2\beta} }}{n}\nn\\
	&+ \ep^2 e^{2T M_n^{2\beta} } T^2 \|{ u}\|^2_{L^\infty ([0,T];H^{2\beta}(\Omega))} +  \Big(\|{u}_T\|_{\tilde V(\Omega)}^2+ T \|g\|^2_{L^\infty([0,T]; \widetilde V(\Omega))}+T a_0^2 \|{u}\|_{L^\infty\lf([0,T];\tilde V(\Omega)\rt)}^2 \Big).
	\end{align}
	Hence
	\begin{align*}
	\mathbb{\bf E}  \|W_{M_n,n}(\cdot,t)-{u}(\cdot,t)\|^2_{L^2(\Omega)} \le e^{(4K(Q_n) +4)T} e^{-2t M_n^\beta } { \Phi} (n,{ u},g,\delta,\ep).
	\end{align*}	
\end{proof}

\subsection{ Error estimate in $H^\beta(\Omega)$}\label{section-H-beta-estimate}
In this subsection, we give error estimate between the regularized solution and the sought solution in higher Sobolev spaces.
\begin{theorem}\label{thm-high-sob-estimate}
	Suppose that  $M_n$
satisfies
	\begin{equation}
	\lim_{n\to +\infty}  \frac{(M_n^{2\beta+1}+M_n^{2\beta})e^{2 M_n^{2\beta} T}}{n}\quad \text{bounded},
	\end{equation}
	and  $\ep$ satisfies the condition in equation \eqref{noiselevel}.
	Let us choose $Q_n$ such that
	\begin{equation}
	\lim_{n \to +\infty}	e^{-2 M_n^{2\beta} t} \exp\Big( \frac{8}{b_0} K^2(Q_n)T \Big) =0,~~t \in (0,T].
	\end{equation}	
	Assume that $g, { u} \in L^\infty\lf([0,T];\tilde V(\Omega)\rt)$, where $\tilde V(\Omega)$ is defined in Lemma \ref{lemma-r-p-bounds}. Then we have
	\begin{align}
	\mathbb{\bf E}  \|W_{M_n,n}(\cdot,t)-{u}(\cdot,t)\|^2_{H^\beta(\Omega)} \le 	e^{-2 M_n^{2\beta} t} \exp\Big( \frac{8}{b_0} K^2(Q_n)(T-t) \Big)  {\bf \Pi} (n,{u},g,\delta,\ep).
	\end{align}
	Where
\begin{align} \label{Pii-0}
	&{\bf \Pi} (n,{ u},g,\delta,\ep)=\left(\pi^2 V_{\max}^2+ \big|\overline C (\delta,u) \big|^2+ \frac{8}{b_0} \pi^2 T^3 \vartheta^2+\frac{8}{b_0}T \big|\overline D (\delta,g)\big|^2 \right) \frac{(M_n^{2\beta+1}+M_n^{2\beta})e^{2 M_n^{2\beta} T}}{n}\nn\\
	&	+\frac{4 e^{2 M_n^{2\beta} T}\ep^2 T^2}{b_0}  \|{ u}\|^2_{L^\infty [0,T];H^{2\beta}(\Omega))}+  \|{u}_T\|_{\tilde V(\Omega)}^2+ \frac{8T}{b_0} \|g\|^2_{L^\infty([0,T]; \widetilde V(\Omega))}+\frac{8T a_0^2}{b_0} \|\textbf{u}\|_{L^\infty\lf([0,T];\tilde V(\Omega)\rt)}^2
	\end{align}
\end{theorem}

{ %\color{red}
\begin{remark}
1.	It is easy to see that when $t>0$ then the error $\mathbb{\bf E}  \|W_{M_n,n}(\cdot,t)-{u}(\cdot,t)\|^2_{H^\beta(\Omega)}$ is of order
	\begin{equation}
		e^{-2 M_n^{2\beta} t} \exp\Big( \frac{8}{b_0} K^2(Q_n)T \Big).
	\end{equation}
One example for  $M_n$ and $Q_n$ for  Theorem \ref{thm-high-sob-estimate} can be found  in Remark \ref{mainremark}. The estimate at $t=0$ is showed by similar argument as in Lemma \ref{mainlemma}, so, we omit it here. \\

2. In  Theorems \ref{thm-time-dependent-1} and \ref{thm-high-sob-estimate} the upper bounds of the approximations are very complex. The reason is that  the models in Section \ref{section3} are  more complex than the ones in Section \ref{section2}.  Indeed, in section \ref{section3}, the time dependent coefficient $a(t)$ is noisy by random coefficients and the source term $F$ is locally Lipschitz continuous. The estimates  in this case are not simple.
\end{remark}
}

\begin{proof}[\bf Proof of Theorem \ref{thm-high-sob-estimate}]
	Assume that $\overline b(t) \ge b_0>0$.
	By taking the inner product of the  two sides of  equality \eqref{difference} with $(-\Delta )^\beta \mathbf{Y}_{M_n,n}$
	one deduces that
	\begin{align*} \label{3J}
	&\frac{1}{2} \frac{d}{dt}  \|\mathbf{ Y}_{M_n,n}(\cdot, t) \|^2_{H^\beta(\Omega)} - \overline  b(t) \Big\|(-\Delta )^\beta \mathbf{ Y}_{M_n,n}\Big\|^2_{L^2(\Omega)}- \rho_n \|\mathbf{  Y}_{M_n,n}(\cdot, t)\|^2_{H^\beta(\Omega)}\nn\\
	& = \underbrace{ \Big<	- e^{\rho_n(t-T)} a_0 {\bf \overline P}_{n,\beta}  \mathbf{ Y}_{M_n,n}, (-\Delta )^\beta \mathbf{ Y}_{M_n,n} \big>_{L^2(\Omega)}}_{=:\widetilde{\mathcal{J}}_{6,n}} + \underbrace{ \Big< e^{\rho_n(t-T)} a_0{\bf \overline R}_{n,\beta} {u}, (-\Delta )^\beta \mathbf{ Y}_{M_n,n} \Big>_{L^2(\Omega)}}_{=:\widetilde{\mathcal{J}}_{7,n}} \nn\\
	&+ \underbrace{\Big< - e^{\rho_n(t-T)}\Big(\overline  a(t)-  a(t)  \Big)(-\Delta )^\beta  { u} , (-\Delta )^\beta \mathbf{Y}_{M_n,n} \Big>_{L^2(\Omega)}}_{=:\widetilde{\mathcal{J}}_{8,n}} \nn\\
	&+\underbrace{\Big< e^{\rho_n (t-T)}\lf[\overline F_{Q_n}(W_{M_n,n}(\cdot,t))- F\big(u(\cdot,t)\big)\rt], -(-\Delta )^\beta \mathbf{\widehat Y}_{M_n,n} \Big>_{L^2(\Omega)}}_{=:\widetilde{\mathcal{J}}_{9,n}}\nn\\
	&+\underbrace{\Big<e^{\rho_n (t-T)}\lf[\overline g_{M_n,n}(\cdot,t))- g(\cdot,t)\rt], (-\Delta )^\beta \mathbf{ Y}_{M_n,n} \Big>_{L^2(\Omega)}}_{=:\widetilde{\mathcal{J}}_{10,n}}	.
	\end{align*}
	For $\widetilde{\mathcal{J}}_{6,n}$, we have
	\begin{align*}
	&|\widetilde{\mathcal{J}}_{6,n}|=\left| \Big<	- e^{\rho_n(t-T)} a_0 {\bf \overline P}_{n,\beta}  \mathbf{ Y}_{M_n,n}, (-\Delta )^\beta \mathbf{ Y}_{M_n,n} \big>_{L^2(\Omega)} \right| \nn\\
	&=e^{\rho_n(t-T)} a_0 \int_{\Omega}  \Big(  \sum_{p < M_n a_0^{-\frac{1}{2\beta}  }}   p^{2\beta} <\mathbf{ Y}_{M_n,n}, \phi_p>\phi_p(x) \Big) \Big( \sum_{p=0}^\infty p^{2\beta} <\mathbf{ Y}_{M_n,n}, \phi_p>\phi_p(x) \Big)dx\nn\\
	&=e^{\rho_n(t-T)} a_0  \sum_{p < M_n a_0^{-\frac{1}{2\beta}  }}   p^{4\beta}<\mathbf{ Y}_{M_n,n}, \phi_p>^2 \nn\\
	&\le M_n^{2\beta} \sum_{p < M_n a_0^{-\frac{1}{2\beta}  }}   p^{2\beta}<\mathbf{ Y}_{M_n,n}, \phi_p>^2 \le  M_n^{2\beta} \|\mathbf{ Y}_{M_n,n}\|_{H^\beta(\Omega)}^2 .
	\end{align*}
	We bound $\widetilde{\mathcal{J}}_{7,n}$ by
	\begin{align*}
	\Big|\widetilde{\mathcal{J}}_{7,n}\Big| &\le \frac{4}{b_0}  e^{2\rho_n(t-T)} a_0^2 \Big\|{\bf \overline R}_{n,\beta}  \mathbf{ Y}_{M_n,n}\Big\|^2_{L^2(\Omega)}+\frac{b_0}{4} \Big\|(-\Delta )^\beta \mathbf{ Y}_{M_n,n}\Big\|^2_{L^2(\Omega)}\nn\\
	&\le \frac{4}{b_0}  a_0^2 e^{-TM_n^{2\beta} } \|{u}\|_{L^\infty\lf([0,T];\tilde V(\Omega)\rt)}^2 +\frac{b_0}{4} \Big\|(-\Delta )^\beta \mathbf{ Y}_{M_n,n}\Big\|^2_{L^2(\Omega)}.
	\end{align*}
	The term $\widetilde{\mathcal{J}}_{8,n}$ is bounded by
	\begin{align*}
	\Big|\widetilde{\mathcal{J}}_{8,n}\Big|
	&\le \frac{4}{b_0}  e^{2\rho_n(t-T)}  \Big(\overline  a(t)-  a(t)  \Big)^2   \lf\| (-\Delta )^\beta  { u} \rt\|_{L^2(\Omega)}^2 +\frac{b_0}{4} \Big\|(-\Delta )^\beta \mathbf{ Y}_{M_n,n}\Big\|^2_{L^2(\Omega)}\nn\\
	&\le  \frac{4}{b_0}   e^{2\rho_n(t-T)}  \Big(\overline  a(t)-  a(t)  \Big)^2  \|{ u}\|^2_{L^\infty( [0,T];H^{2\beta}(\Omega))}+\frac{b_0}{4} \Big\|(-\Delta )^\beta \mathbf{ Y}_{M_n,n}\Big\|^2_{L^2(\Omega)}.
	\end{align*}
	The term $\widetilde{\mathcal{J}}_{9,n}$ is estimated as follows
	\begin{align*}
	\Big|\widetilde{\mathcal{J}}_{9,n}\Big|  &\le \frac{4}{b_0}   e^{2\rho_n(t-T)} \Big\| \overline F_{Q_n}(W_{M_n,n}(\cdot,t))- F\big(u(\cdot,t)\big) \Big\|^2 +\frac{b_0}{4} \Big\|(-\Delta )^\beta \mathbf{ Y}_{M_n,n}\Big\|^2_{L^2(\Omega)}\nn\\
	&\le \frac{4}{b_0}  K^2(Q_n) \|\mathbf{Y}_{M_n,n}(\cdot,t)\|_{L^2(\Omega)}^2 +\frac{b_0}{4} \Big\|(-\Delta )^\beta \mathbf{ Y}_{M_n,n}\Big\|^2_{L^2(\Omega)}\nn\\
	&\le \frac{4}{b_0}  K^2(Q_n) \|\mathbf{Y}_{M_n,n}(\cdot,t)\|_{H^\beta(\Omega)}^2 +\frac{b_0}{4} \Big\|(-\Delta )^\beta \mathbf{ Y}_{M_n,n}\Big\|^2_{L^2(\Omega)}.
	\end{align*}
	where in the latter inequality,  we have noted that for all $v \in H^\beta(\Omega)$ then
	\begin{align}
	\|v\|^2_{H^\beta(\Omega)}= \sum_{p=0}^\infty p^{2\beta} <v, \phi_p>^2 \ge \sum_{p=0}^\infty  <v, \phi_p>^2 = \|v\|^2_{L^2(\Omega)}.
	\end{align}
	The term $\big|\widetilde{\mathcal{J}}_{10,n}\big|$ can be bounded by
	\begin{align} \label{J4}
	\big|\widetilde{\mathcal{J}}_{10,n}\big|&= \Big| \Big<e^{\rho_n (t-T)}\lf[\overline g_{M_n,n}(\cdot,t))- g(\cdot,t)\rt], (-\Delta)^\beta \mathbf{Y}_{M_n,n} \Big>_{L^2(\Omega)}\Big|\nn\\
	& \leq \frac{4}{b_0} e^{2\rho_n (t-T)} \lf\|\overline g_{M_n,n}(\cdot,t))- g(\cdot,t)\rt\|^2_{L^2(\Omega)}+ \frac{b_0}{4} \Big\|(-\Delta )^\beta \mathbf{ Y}_{M_n,n}\Big\|^2_{L^2(\Omega)}.
	\end{align}
	{ From the above observations, we obtain
	\begin{align} \label{3J}
	&\frac{1}{2} \frac{d}{dt}  \|\mathbf{ Y}_{M_n,n}(\cdot, t) \|^2_{H^\beta(\Omega)}\nn\\
	& \geq \overline  b(t) \Big\|(-\Delta )^\beta \mathbf{ Y}_{M_n,n}\Big\|^2_{L^2(\Omega)}+ \rho_n \|\mathbf{  Y}_{M_n,n}(\cdot, t)\|^2_{H^\beta(\Omega)}- M_n^{2\beta}  \Big\| \mathbf{ Y}_{M_n,n}\Big\|_{H^\beta(\Omega)}^2\nn\\
	&\quad \quad \quad-\frac{4 a_0^2}{b_0}  e^{-TM_n^{2\beta} } \|{u}\|_{L^\infty\lf([0,T];\tilde V(\Omega)\rt)}^2 -\frac{b_0}{4} \Big\|(-\Delta )^\beta \mathbf{ Y}_{M_n,n}\Big\|^2_{L^2(\Omega)}\nn\\
	&\quad \quad \quad-\frac{4}{b_0}   e^{2\rho_n(t-T)}  \Big(\overline  a(t)-  a(t)  \Big)^2  \|{ u}\|^2_{L^\infty( [0,T];H^{2\beta}(\Omega))}-\frac{b_0}{4} \Big\|(-\Delta )^\beta \mathbf{ Y}_{M_n,n}\Big\|^2_{L^2(\Omega)}\nn\\
	&\quad \quad \quad-\frac{4}{b_0}  K^2(Q_n) \|\mathbf{Y}_{M_n,n}(\cdot,t)\|_{H^\beta(\Omega)}^2 -\frac{b_0}{4} \Big\|(-\Delta )^\beta \mathbf{ Y}_{M_n,n}\Big\|^2_{L^2(\Omega)}\nn\\
	&\quad \quad \quad-\frac{4}{b_0} e^{2\rho_n (t-T)} \lf\|\overline g_{M_n,n}(\cdot,t)- g(\cdot,t)\rt\|^2_{L^2(\Omega)}- \frac{b_0}{4} \Big\|(-\Delta )^\beta \mathbf{ Y}_{M_n,n}\Big\|^2_{L^2(\Omega)}\nn\\
	&= \left( \overline  b(t) -b_0\right)\Big\|(-\Delta )^\beta \mathbf{ Y}_{M_n,n}\Big\|^2_{L^2(\Omega)}-\frac{4}{b_0}   e^{2\rho_n(t-T)}  \Big(\overline  a(t)-  a(t)  \Big)^2  \|{ u}\|^2_{L^\infty( [0,T];H^{2\beta}(\Omega))}\nn\\
	&\quad \quad \quad-\frac{4}{b_0} e^{2\rho_n (t-T)} \lf\|\overline g_{M_n,n}(\cdot,t)- g(\cdot,t)\rt\|^2_{L^2(\Omega)}-\frac{4 a_0^2}{b_0}  e^{-TM_n^{2\beta} } \|{u}\|_{L^\infty\lf([0,T];\tilde V(\Omega)\rt)}^2\nn\\
	&\quad \quad \quad+\left(  \rho_n- M_n^{2\beta} -\frac{4}{b_0}  K^2(Q_n)\right) \|\mathbf{ Y}_{M_n,n}(\cdot, t) \|^2_{H^\beta(\Omega)}.
	\end{align}
}
By the fact that $\overline  b(t) \ge b_0$, we know that the term $\left( \overline  b(t) -b_0\right)\Big\|(-\Delta )^\beta \mathbf{ Y}_{M_n,n}\Big\|^2_{L^2(\Omega)}$ is non-negative. It follows from \eqref{3J} that
\begin{align}
\frac{d}{dt}  \|\mathbf{ Y}_{M_n,n}(\cdot, t) \|^2_{H^\beta(\Omega)} &+\frac{8 a_0^2}{b_0}  e^{-TM_n^{2\beta} } \|{u}\|_{L^\infty\lf([0,T];\tilde V(\Omega)\rt)}^2\nn\\
&+\frac{8}{b_0}   e^{2\rho_n(t-T)}  \Big(\overline  a(t)-  a(t)  \Big)^2  \|{\bf u}\|^2_{L^\infty( [0,T];H^{2\beta}(\Omega))}\nn\\
&+\frac{8}{b_0} e^{2\rho_n (t-T)} \lf\|\overline g_{M_n,n}(\cdot,t)- g(\cdot,t)\rt\|^2_{L^2(\Omega)}\nn\\
&\ge 2\left(  \rho_n- M_n^{2\beta} -\frac{4}{b_0}  K^2(Q_n)\right) \|\mathbf{ Y}_{M_n,n}(\cdot, t) \|^2_{H^\beta(\Omega)}.
\end{align}
	By taking the integral from $t$ to $T$, we obtain that
	\begin{align}  \label{176}
	&\|\mathbf{Y}_{M_n,n}(\cdot, T)\|^2_{H^\beta(\Omega)}  -\|\mathbf{Y}_{M_n,n}(\cdot, t)\|^2_{H^\beta(\Omega)} \nn\\
	&\quad \quad+ \underbrace{\int_t^T \lf( \frac{8 a_0^2}{b_0}   e^{-2TM_n^{2\beta} }  \|{u}\|_{L^\infty\lf([0,T];\tilde V(\Omega)\rt)}^2  +\frac{8}{b_0}  \Big(\overline  a(s)-  a(s)  \Big)^2  \|{ u}\|^2_{L^\infty( [0,T];H^{2\beta}(\Omega))}\rt) ds }_{:= J_{11}}\nn\\
	&\quad \quad+\underbrace{\int_t^T  e^{2 \rho_n (s-T)} \frac{8}{b_0} \lf\|\overline g_{M_n,n}(\cdot,s))- g(\cdot,s)\rt\|^2_{L^2(\Omega)} ds}_{:=J_{12}}\nn\\
	&\quad \quad \geq \int_t^T \Big(2\rho_n- 2M_n^{2\beta} - \frac{8}{b_0} K^2(Q_n)  \Big)\|\mathbf{Y}_{M_n,n}(\cdot, s) \|^2_{H^\beta(\Omega)} ds.
	\end{align}
	Let us choose $\rho_n=M_n^{2\beta}$, we have that
	\begin{align}
	\|\mathbf{Y}_{M_n,n}(\cdot, t)\|^2_{H^\beta(\Omega)} &\le \frac{8}{b_0} K^2(Q_n)  \int_t^T   \|\mathbf{Y}_{M_n,n}(\cdot, s) \|^2_{H^\beta(\Omega)} ds+ 	\|\mathbf{Y}_{M_n,n}(\cdot, T)\|^2_{H^\beta(\Omega)} \nn\\
	&+ J_{11}+J_{12}.\label{177}
	\end{align}
	
Next we give upper bounds for  the terms $J_{11}$ and $J_{12}$ of 	\eqref{177}.\\
For $J_{11}$, by equation \eqref{a}, we have
\begin{align}
J_{11} &\le \frac{8 a_0^2}{b_0} (T-t)  e^{-2TM_n^{2\beta} }  \|{u}\|_{L^\infty\lf([0,T];\tilde V(\Omega)\rt)}^2 +\frac{8}{b_0} \|{ u}\|^2_{L^\infty ([0,T];H^{2\beta}(\Omega))} \int_t^T \ep^2 |\overline \xi(s)|^2 ds.
\end{align}
Since ${\bf E}|\overline \xi(s)|^2=s $, we have the following estimation
\begin{align}
{\bf E} J_{11}
&\le \frac{8 T a_0^2}{b_0}   e^{-2TM_n^{2\beta} }  \|{u}\|_{L^\infty\lf([0,T];\tilde V(\Omega)\rt)}^2 +\frac{8}{b_0} \|{ u}\|^2_{L^\infty ([0,T];H^{2\beta}(\Omega))} \int_t^T \ep^2 {\bf E}|\overline \xi(s)|^2  ds \nn\\
&\le \frac{8 T a_0^2}{b_0}   e^{-2TM_n^{2\beta} }  \|{u}\|_{L^\infty\lf([0,T];\tilde V(\Omega)\rt)}^2 +\frac{4 \ep^2 T^2}{b_0} \|{ u}\|^2_{L^\infty ([0,T];H^{2\beta}(\Omega))} . \label{178}
\end{align}
For $J_{12}$, by equation \eqref{129}, we get
\begin{align}
{\bf E} J_{12} &\le   \frac{8}{b_0} \int_t^T  {\bf E} \lf\|\overline g_{M_n,n}(\cdot,s))- g(\cdot,s)\rt\|^2_{L^2(\Omega)}ds\nn\\
& \le \frac{8}{b_0} \Bigg[\left( \pi^2 T^2 \vartheta^2+ \big|\overline D (\delta,g)\big|^2 \right)  \frac{M_n+1}{n}+e^{-2TM_n^{2\beta}} \|g\|_{L^\infty([0,T]; \widetilde V(\Omega))} \Bigg] (T-t). \label{179}
\end{align}
Now, we continue to estimate $\|\mathbf{Y}_{M_n,n}(\cdot, T)\|^2_{H^\beta(\Omega)} $ of 	\eqref{176}.
		From \eqref{error}, we obtain	\begin{align}
		\|\mathbf{Y}_{M_n,n}(\cdot, T)\|^2_{H^\beta(\Omega)}& =\|\overline  w_{M_n,n}- { u}_T\|^2_{H^\beta(\Omega)}\nn\\
		 &\le \Big[ \dfrac{1}{n}\sum_{k=1}^n \sigma_k \epsilon_k - \widetilde {G}_{n0} \Big]^2+\sum_{p=1}^{M_n} p^{2\beta} \Bigg[  \dfrac{\pi}{n} \sum_{k=1}^n \sigma_k \epsilon_k \phi_p(x_k) -  \widetilde {G}_{np} \Bigg]^2\nonumber\\\
		&+ \sum_{p=M_n+1}^\infty p^{2\beta} \Big<{ u}_T, \phi_p\Big>^2\nn\\
		&\le\Big[ \dfrac{1}{n}\sum_{k=1}^n \sigma_k \epsilon_k - \widetilde {G}_{n0} \Big]^2 +M_n^{2\beta} \sum_{p=1}^{M_n}  \Bigg[  \dfrac{\pi}{n} \sum_{k=1}^n \sigma_k \epsilon_k \phi_p(x_k) -  \widetilde {G}_{np} \Bigg]^2\nn\\
		&+\sum_{p=M_n+1}^\infty p^{2\beta} \Big<{u}_T, \phi_p\Big>^2.
		\end{align}	
		Using the similar techniques as in the proof of Theorem 3.1,  we get
		\begin{align}
	{\bf E} \|\mathbf{Y}_{M_n,n}(\cdot, T)\|^2_{H^\beta(\Omega)}	&\le   \left(\pi^2 V_{\max}^2+ \big|\overline C (\delta,u) \big|^2 \right) \frac{M_n^{2\beta+1}+M_n^{2\beta}}{n}\nn\\
		&+ e^{-2TM_n^2\beta}\sum_{p=M_n+1}^\infty p^{2\beta } e^{2T p^{2\beta}} \Big<{u}_T, \phi_p\Big>^2 \nn\\
		& \le \left(\pi^2 V_{\max}^2+ \big|\overline C (\delta,u) \big|^2 \right) \frac{M_n^{2\beta+1}+M_n^{2\beta}}{n}+  e^{-2TM_n^{2\beta}}\|{u}_T\|_{\tilde V(\Omega)}^2. \label{181}
		\end{align}
	Combining equations \eqref{177}, \eqref{178}, \eqref{179}, \eqref{181}, we derive  that
	\begin{align*}
	&e^{2 M_n^{2\beta} t}	\mathbb{\bf E}  \|W_{M_n,n}(\cdot,t)-{u}(\cdot,t)\|^2_{H^\beta(\Omega)}  \nn\\
	&\quad \quad \leq \left(\pi^2 V_{\max}^2+ \big|\overline C (\delta,u) \big|^2+ \frac{8}{b_0} \pi^2 T^3 \vartheta^2+\frac{8}{b_0}T \big|\overline D (\delta,g)\big|^2 \right) \frac{(M_n^{2\beta+1}+M_n^{2\beta})e^{2 M_n^{2\beta} T}}{n}\nn\\
	&	+\frac{4 e^{2 M_n^{2\beta} T}\ep^2 T^2}{b_0}  \|{ u}\|^2_{L^\infty [0,T];H^{2\beta}(\Omega))}\nn\\
	&+  \Big(\|{u}_T\|_{\tilde V(\Omega)}^2+ \frac{8T}{b_0} \|g\|^2_{L^\infty([0,T]; \widetilde V(\Omega))}+\frac{8T a_0^2}{b_0} \|{u}\|_{L^\infty\lf([0,T];\tilde V(\Omega)\rt)}^2 \Big)\nn\\
	&\quad \quad +\frac{8}{b_0} K^2(Q_n)  \int_t^T e^{2M_n^{2\beta} s} \mathbb{\bf E}  \|W_{M_n,n}(\cdot,s)-{u}(\cdot,s)\|^2_{H^\beta(\Omega)}ds  .
	\end{align*}	
	Let us	denote
	\begin{align*} \label{Pii}
	&{\bf \Pi} (n,{ u},g,\delta,\ep)\nn\\
	&=\left(\pi^2 V_{\max}^2+ \big|\overline C (\delta,u) \big|^2+ \frac{8}{b_0} \pi^2 T^3 \vartheta^2+\frac{8}{b_0}T \big|\overline D (\delta,g)\big|^2 \right) \frac{(M_n^{2\beta+1}+M_n^{2\beta})e^{2 M_n^{2\beta} T}}{n}\nn\\
	&	+\frac{4 e^{2 M_n^{2\beta} T}\ep^2 T^2}{b_0}  \|{ u}\|^2_{L^\infty [0,T];H^{2\beta}(\Omega))}+  \|{u}_T\|_{\tilde V(\Omega)}^2\nn\\
	&+ \frac{8T}{b_0} \|g\|^2_{L^\infty([0,T]; \widetilde V(\Omega))}+\frac{8T a_0^2}{b_0} \|{u}\|_{L^\infty\lf([0,T];\tilde V(\Omega)\rt)}^2
	\end{align*}
	then we have
	\begin{align*}
	e^{2 M_n^{2\beta} t}	\mathbb{\bf E}  \|W_{M_n,n}(\cdot,t)-{u}(\cdot,t)\|^2_{H^\beta(\Omega)}  &\le {\bf \Pi} (n,{ u},g,\delta,\ep)\nn\\
	&
	+\frac{8}{b_0} K^2(Q_n)  \int_t^T e^{2M_n^{2\beta} s} \mathbb{\bf E}  \|W_{M_n,n}(\cdot,s)-{u}(\cdot,s)\|^2_{H^\beta(\Omega)}ds  .
	\end{align*}	
	Using Gronwall's inequality, we obtain
	\begin{align*}
	e^{2 M_n^{2\beta} t}	\mathbb{\bf E}  \|W_{M_n,n}(\cdot,t)-{u}(\cdot,t)\|^2_{H^\beta(\Omega)} \le \exp\Big( \frac{8}{b_0} K^2(Q_n)(T-t) \Big)  {\bf \Pi} (n,{ u},g,\delta,\ep).
	\end{align*}
	This implies that
	\begin{align*}
	\mathbb{\bf E}  \|W_{M_n,n}(\cdot,t)-{u}(\cdot,t)\|^2_{H^\beta(\Omega)} \le 	e^{-2 M_n^{2\beta} t} \exp\Big( \frac{8}{b_0} K^2(Q_n)(T-t) \Big)  {\bf \Pi} (n,{ u},g,\delta,\ep).
	\end{align*}
\end{proof}

\noindent {\bf Acknowledgments}
This work  is supported by Vietnam National University-Ho Chi Minh City (VNU-HCM) under the Grant number B2017-18-03.


\begin{thebibliography}{99}	

\bibitem{BBC} J. V. Beck, B. Blackwell, and  St. C. R. Clair. {Inverse Heat Conduction, Ill--Posed Problems}, Wiley--Interscience, New York, 1985.


\bibitem{JB} J. Bear. {Dynamics of Fluids in Porous Media}, Elsevier, New York, 1972.


\bibitem{Bissantz} N. Bissantz and H. Holzmann. \emph{ Statistical inference for inverse problems,} Inverse Problems, No. 24, 034009, 1--17, 2008.


\bibitem{CSH}A. S.  Carasso, J. G. Sandersona and  J.M. Hyman. \emph{Digital Removal of Random Media Image Degradations by Solving the Diffusion Equation Backwards in Time}, SIAM Journal on Numerical Analysis, Vol. 15 No. 2, 344--367, 1978.


\bibitem{Cavalier} L. Cavalier. \emph{ Nonparametric statistical inverse problems}, Inverse Problems, No. 24, 034004, 19, 2008.


\bibitem{Nane} Z.Q. Chen,  M.M. Meerschaert, and   E. Nane. \emph{ Space-time fractional diffusion on bounded domains}  J. Math. Anal. Appl. 393 (2012), no. 2, 479--488.




%\bibitem{r4} Eubank, R. L., \emph{Nonparametric Regression and Spline Smoothing 2nd edn}, New York: Dekker, 1999.

\bibitem{evan} L. C. Evans. \emph{ Partial differential equations,} American Mathematical Society Provi-
dence, 1998.

\bibitem{Koba}  H. Koba,  H. Matsuoka. \emph{ Generalized quasi-reversibility method for a backward heat equation with a fractional Laplacian}  Analysis (Berlin) 35 (2015), no. 1, 47–-57.

\bibitem {Hohage} C. K\"onig,  F. Werner,  T. Hohage. \emph{ Convergence rates for exponentially ill-posed inverse problems with impulsive noise.} SIAM J. Numer. Anal. 54 (2016), no. 1, 341–-360

	\bibitem{Lassas}	 H. Kekkonen,  M. Lassas, S. Siltanen \emph{ Analysis of regularized inversion of data corrupted by white Gaussian noise}  Inverse Problems 30 (2014), no. 4, 045009, 18 pp.


\bibitem{r5} A. Kirsch. \emph{An Introduction to the Mathematical Theory of  Inverse Problems}, Springer, 1996.


\bibitem{r7} B.  Mair and F. H. Ruymgaart. \emph{Statistical estimation in Hilbert scale,} SIAM J. Appl. Math., No. 56, 1424--1444, 1996.


\bibitem{Minh} N.D. Minh, T.D. Khanh, N.H. Tuan, and D.D. Trong. \emph{ A Two Dimensional Backward Heat Problem With Statistical Discrete Data }, https://arxiv.org/abs/1606.05463.

\bibitem{PTN} Phan Thanh Nam.  \emph{An approximate solution for nonlinear backward parabolic equations}, Journal of Mathematical Analysis and Applications, Vol. 367 No. 2, 337--349, 2010.


\bibitem{LEP} L.E. Payne, L. E. \emph{Improperly Posed Problems in Partial Differential Equations}, SIAM, Philadelphia, PA, 1975.


\bibitem{RHN} M. Renardy, W.J. Hursa and  J. A. Nohel.  \emph{Mathematical Problems in Viscoelasticity}, Wiley, New York, 1987.


\bibitem{SK} T. H. Skaggs and Z. J. Kabala. \emph{Recovering the history of a groundwater contaminant plume: Method of quasi--reversibility}, Water Resources Research, Vol. 31 No. 11, 2669--2673, 1995.


\bibitem{Trong} D.D. Trong, T.D. Khanh, N.H. Tuan, N.D. Minh. \emph{Nonparametric regression in a statistical modified Helmholtz equation using the Fourier spectral regularization.} Statistics 49 (2015), no. 2, 267-–290

\bibitem{Tuan1} N. H. Tuan and E. Nane. \emph{ Inverse source problem for time fractional diffusion with discrete random noise}. Statistics and  Probability Letters. Volume 120, January 2017, Pages 126–-134.


\bibitem{Tuan} N.H. Tuan, L.D. Thang, D. Lesnic. \emph{ A new general filter regularization method for Cauchy problems for elliptic equations with a locally Lipschitz nonlinear source} J. Math. Anal. Appl. 434 (2016), no. 2, 1376--1393.


	\bibitem{Tuan2} N.H. Tuan. \emph{ Stability estimates for a class of semi-linear ill-posed problems} Nonlinear Anal. Real World Appl. 14 (2013), no. 2, 1203–-1215.
	
	\bibitem{Tuan3}	 N.H. Tuan  and D.D. Trong. \emph{ On a backward parabolic problem with local Lipschitz source} J. Math. Anal. Appl. 414 (2014), no. 2, 678–692.

\end{thebibliography}
\end{document}